\documentclass[reqno,a4paper,11pt]{amsart}
\usepackage{hyperref}
\usepackage{tikz}
\usepackage[active]{srcltx}
\usepackage[curve]{xypic}
\usepackage{graphicx}

\usepackage{color}

\usepackage[active]{srcltx}
\usepackage{epsfig,amsmath}
\usepackage[english]{babel}
\usepackage{amsmath,amsfonts,amssymb,amscd,color,epsfig,amsthm}

\newtheorem{newthm}{Theorem}
\newtheorem{theorem}{Theorem}[section]
\newtheorem{lemma}[theorem]{Lemma}
\newtheorem{proposition}[theorem]{Proposition}
\newtheorem{corollary}[theorem]{Corollary}

\newtheorem{definition}[theorem]{Definition}

\theoremstyle{remark}
\newtheorem{example}{\bf Example}[section]

\theoremstyle{plain}

\newtheorem{remark}[theorem]{Remark}

\numberwithin{equation}{section}

\newcommand{\wt}{\widetilde}

\def\cbar{\overline{\C}}
\def\AAA{{\cal A}}

\def\EEE{{\cal E}}
\def\FFF{{\cal F}}

\def\JJJ{{\cal J}}

\def\RRR{{\cal R}}
\def\SSS{{\cal S}}

\def\g{\gamma}
\def\G{\Gamma}

\def\x{\xi}

\def\smm{\smallsetminus}

\def\R{\mbox{$\mathbb R$}}

\def\C{\mbox{$\mathbb C$}}
\def\T{\mbox{$\mathbb T$}}

\def\D{\mathbb D}
\def\Q{\mathbb Q}

\def\Z{\mbox{$\mathbb Z$}}

\def\lv{ \left(\begin{matrix} }
 \def\rv{\end{matrix}\right)}

\def\Th{\mathbf{\Theta}}

\def\cal{\mathcal}

\def\dw{{\dw}}

\def\ds{\displaystyle}

\newcommand{\mylabel}[1]{\label{#1}}

\newcommand{\REFEQN}[1] { \begin{equation}\mylabel{#1} }
\newcommand{\ENDEQN}{\end{equation}}
\newcommand{\REFTHM}[1] { \begin{theorem}\mylabel{#1} }
\newcommand{\ENDTHM}{\end{theorem}}
\newcommand{\REFNTH}[1] { \begin{newthm}\mylabel{#1} }
\newcommand{\ENDNTH}{\end{newthm}}
\newcommand{\REFPROP}[1]{\begin{proposition}\mylabel{#1} }
\newcommand{\ENDPROP}{\end{proposition} }
\newcommand{\REFLEM}[1]{\begin{lemma}\mylabel{#1} }
\newcommand{\ENDLEM}{\end{lemma} }
\newcommand{\REFCOR}[1]{\begin{corollary}\mylabel{#1} }
\newcommand{\ENDCOR}{\end{corollary} }

\def\smm{ {\setminus}}
\def\ds{\displaystyle }

\def\pf{postcritically-finite }

\def\mystrut{{\rule[-2ex]{0ex}{4.5ex}{}}}

\def\ov{\overline}

\def\T{{\mathbb T}}

\usepackage{tikz}
\usetikzlibrary{arrows}
\tikzstyle{every picture}=[> = to]
\tikzset{cdlabel/.style={execute at begin node=$\scriptstyle,execute at end node=$}}
\tikzset{implication/.style={double equal sign distance, -implies}}
\tikzset{biimplication/.style={double equal sign distance, implies-implies}}

\title{On Thurston's core entropy algorithm}
\author{Yan Gao}
\date{\today}
\begin{document}
\maketitle
\begin{center}
Dedicated to the memory of Tan Lei
\end{center}
\begin{abstract}
The core entropy of polynomials, recently introduced by W. Thurston, is a dynamical invariant extending topological entropy for real maps to complex polynomials, whence providing a new tool to study the parameter space of polynomials. The base is a combinatorial algorithm allowing for the computation of the core entropy given by Thurston, but without supplying a proof. In this paper, we will describe his algorithm and prove its validity.

\vspace{0.1cm}

\vspace{0.1cm}

\noindent{\bf AMS(2010) Subject Classification}: 37B40, 37F10, 37F20.
\end{abstract}
\section{Introduction}
\subsection{Background of the paper}

 During the last year of his life, William P. Thurston developed a theory of degree
$d$-invariant laminations , a tool that he hoped would lead to what he called a
``qualitative picture of (the dynamics of) degree $d$  polynomials''. Thurston discussed
his research on this topic in his seminar at Cornell University and was in the
process of writing an article, but he passed away before completing the
manuscript.  Several people have set out to write supplementary material to the manuscript,  based on what they learned from him throughout his seminar and email
exchanges with him. The outcome is the manuscript \cite{TG}.

The present article is part of the author's Ph.D thesis. The initial purpose, as suggested by my supervisor Tan Lei,  was to fill in the empty section ``Hausdorff dimension and growth rate''  of W. Thurston's
manuscript, as part of the supplementary material in \cite{TG}.  However, as the research developed,  the scope of the work exceeded largely what we had expected. In the end, we have decided to put
a complete treatment of the quadratic case in \cite{TG}, leaving the general case into a series of independent writings.

The present article  is the first writing on the general setting.  We will prove that Thurston's ingenuous entropy algorithm on critical portraits gives the core entropy of complex polynomials.  In forthcoming articles, we will show that there are many different combinatorial models that encode the core entropy. Most models were suggested by W. Thurston. As we have shown (in \cite{TG}) in the quadratic case,
there are a dozen of such models, including the Hausdorff dimension of various objects in various spaces, and the growth rate of various dynamical systems.

\subsection{An introduction to this article}


 Recall that given a continuous map $f$ acting on a compact set $X$, its {topological entropy} $h(X,f)$
is a quantity that measures the complexity growth of the induced dynamical system. It is essentially defined as the growth rate of the number
of itineraries under iteration (see \cite{AKM}).

 The \emph{core entropy} of complex polynomials, to be explained below,  was introduced by W. Thurston around 2011 in order to develop a ``qualitative picture'' of the parameter space of degree $d$ polynomials. In the quadratic case, a rich variety of results about the parameter space were known, see for example \cite{DH, L}. But in the higher degree case ($d\geq 3$), our overall understanding has remained sketchy and unsatisfying.

Let $d\geq2$ be an integer, and $f$  a complex polynomial of degree $d$.
A point $c\in\C$ is called a \emph{critical point} of $f$ if $f'(c)=0$. The \emph{critical set} $ {\rm crit}(f)$  is defined to be $${\rm crit}(f)=\{c\in\C\mid f'(c)=0\},$$ and the \emph{postcritical set}  ${\rm post}(f)$ is defined to be
\begin{equation*}
  {\rm post}(f)=\ov{\{f^n(c) : {c\in{\rm crit}(f)}, {n\ge 1} \}}.
\end{equation*}
In many cases, there exists a $f$-forward
invariant, finite topological tree that contains ${\rm post}(f)$, called the {\it Hubbard tree} of $f$, which captures all essential dynamics of the polynomial. In particular, this tree exists if $f$ is \emph{\pf}, i.e., $\#{\rm post}(f)<\infty$ (see Section \ref{hubbard-tree}).

\begin{definition}[core entropy]
The \emph{core entropy} of $f$, denoted by $h(f)$, is defined as the topological entropy of the restriction of $f$ to its Hubbard tree, when the tree exists.
\end{definition}

From this definition, we see that the core entropy extends the topological entropy of real
polynomials to  complex polynomials, where the invariant real segment in the real polynomial case is replaced by an invariant tree, known as the Hubbard tree.
Hence, the core entropy yields a new way to study the parameter space
of complex polynomials. A fundamental tool in this direction is an effective algorithm allowing for the computation of the core entropy.

Let $f$ be a \pf polynomial of degree $d\geq2$ and let
$H_f$ denote its Hubbard tree.
The simplest way to compute the entropy of $f:H_f\to H_f$ is to write the {\it incidence matrix} for the {\it Markov map} $f$
acting on $H_f$, and take the logarithm of its leading eigenvalue (see Section \ref{entropy-result}). However, this method
requires knowledge of the topology of $H_f$, and is thus difficult to  realize on a computer.

To avoid knowing the topology of the Hubbard tree and the action of $f$ on it, W. Thurston developed a purely combinatorially algorithm (without supplying a proof) using the combinatorial data {\it critical portraits}. The concept of  critical portraits  and this  entropy algorithm will be exhaustively explained in Sections \ref{section-critical-portrait} and  \ref{algorithm2} respectively. Roughly speaking, a \pf polynomial $f$ induces a finite collection
of finite subsets of the unit circle
\[\Th:=\mathbf{\Theta}_f=\{\Theta_1,\ldots, \Theta_s\},\]
called a \emph{weak critical marking} of $f$ (see Definition \ref{weak-critical-marking}). The algorithm takes $\mathbf{\Theta}$ as
 input, constructs a non-negative matrix $A$ (bypassing $f$ and $H_f$), and provides as output its Perron-Frobenius leading eigenvalue $ \rho(\mathbf{\Theta})$. Using this algorithm, Thurston draw a picture of the core entropy of quadratic polynomials as a function of the external angle $\theta$ (see Figure \ref{Thurston-plot}).
 \begin{figure}[htbp]\centering
\includegraphics[width=11cm]{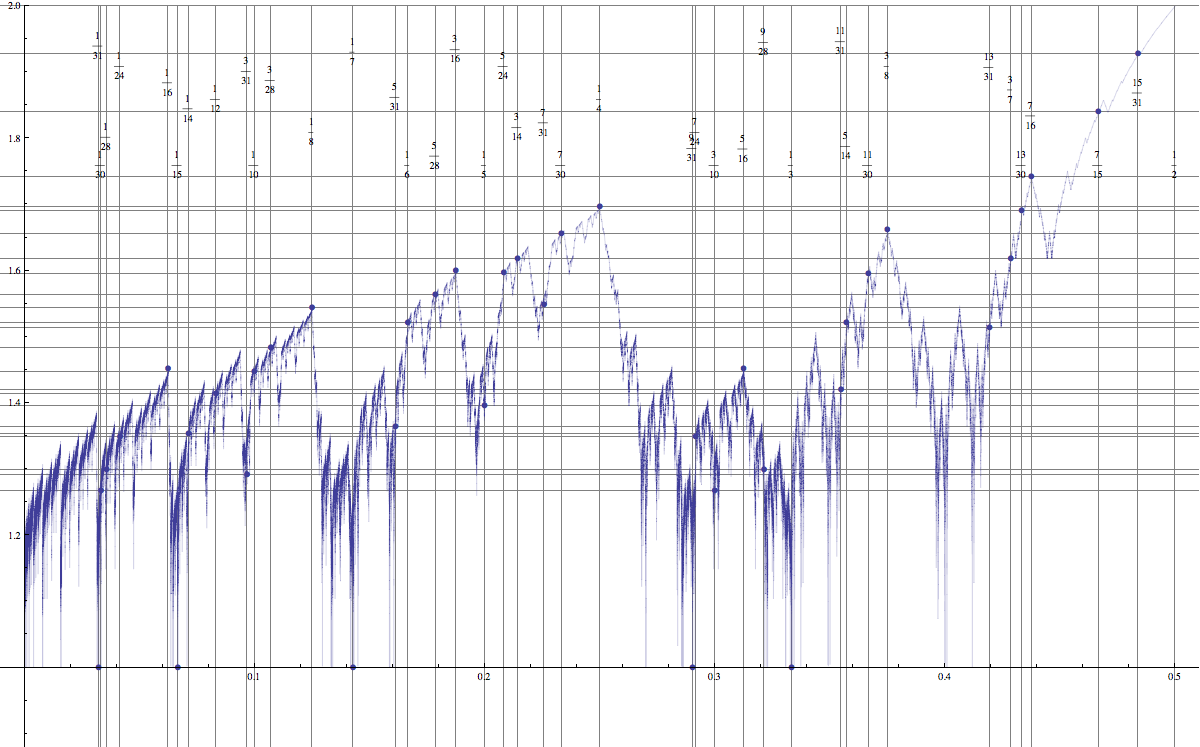}
\caption{The core entropy of quadratic polynomials, drawn by Thurston}\label{Thurston-plot}
\end{figure}

The validity of Thurston's algorithm in the quadratic case was proven by Y. Gao-L. Tan (\cite{TG}) and W. Jung (\cite{Jung}). Based on this algorithm, Tiozzo proved the {\bf continuity conjecture} of Thurston \cite{Ti} (Dudko and Schleicher \cite{DS} give an alternative proof of the conjecture without using this algorithm): Given $\theta\in \Q/\Z$, the parameter ray of angle $\theta$ determines a unique \pf quadratic polynomial $f_{c_\theta}=z^2+c_\theta$. Let $h(\theta)$ denote the core entropy of $f_{c_\theta}$.

{\noindent\bf Theorem} (Thurston, Dudko-Schleicher, Tiozzo).
{\it The entropy function $h:\Q/\Z\to\R$ extends to a continuous function $\R/\Z\to \R$.}


Analogously, in order to clarify
the connectedness locus (the generalization of the Mandelbrot set) for degree $d$ ($d\geq3$) polynomials from the point of view of the core entropy, for example the continuity conjecture in the higher degree case, one should first verify the validity of Thurston's entropy algorithm in the general case. The purpose of this paper is to prove this point. We establish the following main theorem. This is the first step towards Thurston's program.
\begin{theorem}\label{Thurston-algorithm}
Let $f$ be a \pf polynomial, and $\mathbf{\Theta}$  a weak critical marking of $f$. Let $\rho(\Th)$ be the output of Thurston's entropy algorithm. Then $\log\rho(\mathbf{\Theta})$ equals the core entropy of $f$, i.e., $\log\rho(\mathbf{\Theta})=h(H_f,f)$.
\end{theorem}

The organization of the manuscript is as follows.  In Sections 2 and 3, we recall some preliminary definitions and results about  topological entropy and Hubbard trees that will be used below. In Section 4, we give a detailed description of the weak critical markings of \pf polynomials and their induced partitions in the dynamical plane. We then introduce  Thurston's entropy algorithm in Section 5, and prove the main theorem (Theorem \ref{Thurston-algorithm}) in Section 6.

{\noindent \bf Acknowledgement.} I wish to express my sincere appreciation to L. Tan for leading  me to this interesting topic, and offering constant encourage and very useful suggestions during the writing of the paper. Without her help, the paper would not appear.  I also would like to thank R. H. Henrik for the comments on a preliminary version and J. S. Zeng for the very useful discussion and suggestions. Some pictures are provided by J. S. Zeng. The author is partially supported by  NSFC grant no. 11501383.

\section{Basic results about topological entropy}\label{entropy-result}


We will not use the general definition of the topological entropy in the paper (see \cite{AKM}). Instead, we summarize some basic results about the topological entropy that will be applied below.

Let $f:X\to X$ be a continuous map. We denote by $h(X,f)$ the topological entropy of $f$ on $X$.
The following three propositions can be found in \cite{Do}.

\REFPROP{Do2}
If $X=X_1\cup X_2$, with $X_1$ and $X_2$ compact, $f(X_1)\subset X_1$ and $f(X_2)\subset X_2$, then $h(X,f)=\sup\bigl(h(X_1,f),h(X_2,f)\bigr)$. \ENDPROP

\REFPROP{Do3}
Let $Z$ be a closed subset of $X$ such that $f(Z)\subset Z$. Suppose that for any $x\in X$, the distance of $f^n(x)$ to $Z$ tends to $0$, uniformly on any compact set in $X-Z$. Then $h(X,f)=h(Z,f)$. \ENDPROP

\REFPROP{Do4} Assume that $\pi$ is a  surjective semi-conjugacy $$\begin{array}{rcl}Y &\xrightarrow[]{\ q\ }  &Y\\ \pi\Big\downarrow &&\Big\downarrow \pi \vspace{-0.1cm} \\ X &  \xrightarrow[]{\ f\ } & X\vspace{-0.1cm}.\end{array}$$
Then $h(X,f)\leq h(Y,q)$. Furthermore, if  $\ds \sup_{x\in X}\#\pi^{-1}(x)< \infty$ then  $ h(X,f)=h(Y,q)$.\ENDPROP

In this paper, we mainly use the topological entropy for real maps of dimension one.

A (finite topological) \emph{graph} $G$ is a compact  Hausdorff space which contains a
finite non-empty set $V_G$ (the set of vertices), such that every connected component
of $G\setminus V_G$ is homeomorphic to an open interval of the real line.
Since any graph can
be embedded in $\R^3$ (see \cite{Moi}), we will consider each graph endowed with the
topology induced by the topology of $\R^3$.

\begin{definition}[monotone map]\label{monotone}
Let $X,Y$ be topological spaces and $\phi:X\to Y$ be continuous. Then $\phi$ is said to be \emph{monotone} if  $\phi^{-1}(y)$ is  connected for every $y\in Y$.
\end{definition}

The following fact will be repeatedly used in the paper. Refer to e.g.\ \cite[A.13]{BM} for a proof.

\begin{proposition}\label{tree-map}
Let $X$ be a topological space, and $\phi:[0,1]\to X$  a monotone map. Then the image $\phi([0,1])$ is either a point or an \emph{arc}, i.e., a homeomorphic image of a closed interval of the real line.
\end{proposition}

\begin{definition}[Markov graph map]\label{def:markov}
Let $G$ be a finite graph with vertex set $V_G$. A continuous map $f:G\to G$ is called \emph{Markov} if there is a finite subset $A$ of $G$ containing $V_G$ such that $f(A)\subset V_G$ and $f$ is monotone on each component of $G\setminus A$.
\end{definition}

Let $f:G\to G$ be a Markov graph map. By the definition,  an edge of $G$ is mapped either to a vertex of $G$ or the union of several edges of $G$.
Enumerate the edges of  $G$  by $e_i$ , $i=1,\cdots,k$. We then obtain an \emph{incidence matrix} $D_{(G,f)}=(a_{ij})_{k\times k}$ of $(f,X)$
 such that  $a_{ij}=\ell$ if $f(e_i)$ covers $e_j$ precisely $\ell$ times. Note that choosing  different enumerations of the edges gives rise to conjugate incidence matrices, so in particular, the eigenvalues are independent of the choices.

 Denote by $\rho$ the largest non-negative eigenvalue of $D_{(G,f)}$. By the Perron-Frobenius theorem  such an eigenvalue exists
  and equals
the growth rate of $\|D_{(G,f)}^n\|$ for any matrix norm.

The following result is classical (see \cite{AM,MS}):
\begin{proposition}\label{entropy-formula}
The topological entropy $h(G,f)$ is equal to $0$ if $D_{(G,f)}$ is nilpotent, i.e., all eigenvalues of $D_{(G,f)}$ are zero; and equal to $\log\rho$ otherwise.
\end{proposition}

A special and important type of graph is the (topological) \emph{tree}, which is a connected  graph without cycles.
 A point $p$ of a tree $T$  is called an \emph{endpoint} if $T\smm \{p\}$ is connected, and called a \emph{ branch point} if  $T\smm \{p\}$ has at least $3$ connected components. For any two points $p,q\in T$, there is a unique  arc in $T$ joining $p$ and $q$. We denote this arc  by $[p,q]_T$.

\section{Postcritically-finite polynomials and its Hubbard tree}

\subsection{Postcritically-finite polynomials} \label{hubbard-tree}

Let $f$ be a \pf polynomial, i.e., such that each of its critical points has a finite (and hence periodic or preperiodic) orbit under the action of $f$.
By classical results of Fatou, Julia, Douady and Hubbard,  the filled Julia  set $K_f=\{z\in \C\mid f^n(z)\not\to \infty\}$  is  compact, connected, locally connected  and locally arc-connected. These properties also hold for the Julia set $J_f:=\partial K_f$.
The Fatou set $F_f:=\cbar\smm J_f$ consists of one unbounded component $U(\infty)$ which is equal to the basin of attraction of $\infty$, together with at most countably many bounded components constituting the interior of $K_f$. Each of the sets $K_f, J_f, F_f$ and $U(\infty)$ is fully invariant by $f$; each Fatou component is (pre)periodic (by Sullivan's non-wandering domain theorem, or by the hyperbolicity of the map); and each periodic Fatou component cycle contains at least one critical point of $f$ (including $\infty$).

As a consequence, for $f$ a \pf polynomial,
there is a system of Riemann mappings $$\Big\{\phi_U: \D\to U\,\Big|\, U \text{ Fatou component}\Big\}$$ satisfying that
each extends to a continuous map on the closure $\overline{\D}$, so that  the following diagram commutes for all $U$:
\begin{equation*}
 \begin{tikzpicture}
   \matrix[row sep=0.8cm,column sep=2.4cm] {
     \node (Gammai) {$ \overline \D $}; &
       \node (Gamma) {$\overline \D$}; \\
     \node (S2i) {$\overline U$}; &
       \node (S2) {$\overline{f(U)}$,}; \\
   };
   \draw[->] (Gamma) to node[auto=left,cdlabel] {\phi_{f(U)}} (S2);
   \draw[->] (S2i) to node[auto=right,cdlabel] {f} (S2);
   \draw[->] (Gammai) to node[auto=left,cdlabel] {\text{ power map }z^{d_{\tiny\mbox{$U$}}}} (Gamma);
   \draw[->] (Gammai) to node[auto=right,cdlabel] {\phi_U} (S2i);
 \end{tikzpicture}
 \end{equation*}
where $d_U$ denotes the degree of $f$ on $U$. The image $\phi_U(0)$ is called the \emph{center} of the Fatou component $U$. It is easy to see that any center is mapped to a critical periodic point by some iterations of $f$.
 On every periodic Fatou component $U$, including $U(\infty)$, the map $\phi_U$ realizes a conjugacy between a power map and the first return map on
$U$. The image in $U$ under $\phi_U$ of \textsc{closed}, \emph{resp.} \textsc{open}, radial lines in $\ov{\D}$ are, by definition,  \emph{internal rays} of $U$ if $U$ is bounded, \emph{resp.} \emph{external rays} if $U=U(\infty)$.
Since a power map sends a radial line to a radial line, the polynomial $f$ sends an internal/external ray to an internal/external ray.

 If $U$ is a bounded Fatou component, then  $\phi_U:\overline \D\to \overline U$ is a homeomorphism and thus every boundary point of $U$ receives exactly one internal ray from $U$. This is in general not true for $U(\infty)$, where several external rays may land at a common boundary point.
For any $\theta\in \R/\Z$, we use $\RRR_f(\theta)$ or simply $\RRR(\theta)$ to denote the image by $\phi_{U(\infty)}$ of the radial ray $\{re^{2\pi i \theta}, 0<r<1\}$ and will call it the  \emph{external ray of angle $\theta$}. We also use $\g(\theta)=\phi_{U(\infty)}(e^{2\pi i \theta})$ to denote the \emph{landing point} of the ray $\RRR(\theta)$. 
If $\RRR(\theta)$ lands at the boundary of a Fatou component $U$, then there is a unique internal ray of $U$ that joins the center of $U$ and the landing point of $\RRR(\theta)$. We denote this internal ray by  $r_{U}(\theta)$, and call the ray
\begin{equation}\label{extended-ray}
\EEE_U(\theta):=r_U(\theta)\cup \RRR(\theta)
\end{equation}
the \emph{extended ray of angle $\theta$ at $U$}.

\begin{definition}[supporting rays]\label{support-ray} We say that an external ray $\RRR(\theta)$ \emph{ supports} a bounded Fatou component $U$
 if
\begin{enumerate}
\item the ray lands at a boundary point $q$ of $U$, and
\item there is a sector based at $q$ delimited by $\RRR(\theta)$ and the internal ray of $U$ landing at $q$ such that
the sector does not contain other external rays landing at $q$. 
\end{enumerate}
\end{definition}


 \subsection{The Hubbard trees of postcritically finite polynomials}

 The material of this part comes from \cite[Chpter 2]{DH} and \cite[Chpter I]{Poi2}.

Let $f$ be a \pf polynomial. Then
any pair of points in the closure of a bounded Fatou component  can be joined in a unique way by a Jordan arc consisting of (at most two) segments of internal rays. We call such arcs \emph{regulated} (following Douady and Hubbard).
Since $K_f$ is arc connected, given two points $z_1, z_2\in K_f$, there is an arc $\gamma: [0,1]\to K_f$ such that $\gamma(0)=z_1$ and $\gamma(1)=z_2$. In general, we will not distinguish  between the map  $\g$ and its image. It is proved in \cite{DH} that such arcs can be chosen in a unique way so that  the intersection with the closure of any Fatou component is regulated. We still call such arcs regulated and denote them by $[z_1,z_2]$.
We say that a subset $X\subset K_f$ is \emph{allowably connected} if for every $z_1,z_2\in X$ we have $[z_1,z_2]\subset X$.
We define the \emph{regulated hull}
of  a subset $X$ of $ K_f$ to be  the minimal closed allowably connected subset of $K_f$ containing $X$.

\begin{proposition}[\cite{DH}. Proposition 2.7]
Let $f$ be a \pf polynomial. For a collection of $z_1,\ldots,z_n$ finitely many points in $K_f$, their regulated hull   is a finite tree with endpoints in $\{z_1,\ldots,z_n\}$.
\end{proposition}

The \emph{Hubbard tree} of $f$, denoted by $H_f$,  is defined to be the regulated hull  in $K_f$ of the finite set ${\rm post}(f)$.
Its vertex set $V_{H_f}$ is defined to be the union of ${\rm post}(f)$, ${\rm crit}(f)\cap H_f$ and the branch points of $H_f$.
The following is a well-known result (see \cite[Section I.1]{Poi2}).

\begin{proposition}\label{Hubbard-tree} Any \pf  polynomial $f$ maps each edge of $H_f$ homeomorphically onto the union of  edges of $H_f$. Consequently, $f:H_f\to H_f$ is a Markov map.
\end{proposition}

Using Proposition \ref{entropy-formula}, we immediately get the following result.
 \begin{corollary}
The topological entropy of $f$ acting on $H_f$ equals to the logarithm of the spectral radius of the incidence matrix $D_{(H_f,f)}$, i.e., $h(H_f,f)=\log\rho(D_{(H_f,f)})$.
\end{corollary}

\section{Weak critical markings and the induced partitions}

\subsection{Weak critical markings of \pf polynomials}\label{section-critical-portrait}

For each \pf polynomial, we will define in this section  a collection of combinatorial data from the rays landing at its  critical points and on the critical Fatou components, called a \emph{weak critical marking} of the polynomial. 
This concept is a generalization of the well-known one: critical markings of \pf polynomials (see \cite[Section I.2.1]{Poi1}), which was first introduced in \cite{BFH} to classify the strictly preperiodic polynomials as dynamical systems, and then extended by Poirier \cite{Poi1} to the general case (including periodic critical points).


Let $f$ be a \pf  polynomial of degree $d$. We first define $\Theta(U)$  as follows for each  \emph{critical Fatou component} $U$, i.e., a Fatou component containing a critical point. 
Denote $\delta_U=\text{deg}(f|_U)$.
\begin{itemize}
\item In the case that $U$ is a periodic Fatou component, let
\[U\mapsto f(U)\mapsto\cdots\mapsto f^n(U)=U \]
be a critical Fatou cycle of period $n$. We will construct the associated set $\Theta(U')$ for every critical Fatou component $U'$ in this cycle simultaneously.
Let $z\in\partial U$ be a \emph{root} of $U$, i.e., a periodic point  with period less than or equal to $n$. Such $z$ must exist because one can choose it as the landing point of an internal ray of $U$ fixed by $f^n$.  Note that this choice
naturally determines a root $f^k(z)$ for each Fatou component $f^k(U)$ for $k\in\{0,\ldots,n-1\}$, which is called the \emph{preferred root} of $f^k(U)$.
Let $U'$ be a critical Fatou component in the cycle and $z'$  its preferred root.  Consider a supporting ray $\RRR(\theta)$ for this component $U'$ at $z'$. We define $\Theta(U',z',\theta)$ the set of arguments of the $\delta_{U'}$ supporting rays for the component $U'$ that are inverse images
of $f(\RRR(\theta))$. There are finitely many such set $\Theta(U',z',\theta)$ according to the different choices of the roots $z$ of $U$ and the arguments $\theta$. We pick one of them and simply denote it by $\Theta(U')$.
\item In the case that $U$ is a strictly preperiodic Fatou component, let $n$ be the minimal number such that $f^n(U)$ is a critical Fatou component. And let  $z\in \partial U$ be a point which is  mapped by $f^n$ to the point $\g(\eta)$, the landing point of $\RRR(\eta)$, with $\eta\in \Theta(f^n(U))$.  Consider a supporting ray $\RRR(\theta)$ for the component $U$ at $z$, and define $\Theta(U,z,\theta)$ to be the set of arguments of the $\delta_{U}$ supporting rays for $U$ that are inverse images
of $f(\RRR(\theta))$.  There are  finitely many such $\Theta(U,z,\theta)$ according to the different choices of such $z\in \partial U$ and the arguments $\theta$. We pick one of them and simply denote it by $\Theta(U)$.
\end{itemize}

\begin{definition}\label{weak-critical-marking}
Let $f$ be a \pf polynomial, with  $U_1,\ldots,U_n$ the pairwise disjoint critical Fatou components.  A finite collection of finite subsets of the unit circle
\begin{equation}\label{eq:weak-critical-marking}
\mathbf{\Theta}=\mathbf{\Theta}_f:=\{\Theta_1(c_1),\ldots,\Theta_m(c_m);\Theta(U_1),\ldots,\Theta(U_n)\}
\end{equation}
is called a \emph{weak critical marking} of $f$ if
\begin{enumerate}
\item each $\Theta(U_k)$ is defined as above for $k\in\{1,\ldots,n\}$;
\item the union of $c_1,\ldots, c_m$ (not necessarily pairwise disjoint) equals  the union of the critical points of $f$ in $J_f$;
\item each $\Theta_j(c_j)$ consists of at least two angles such that the external rays with these angles land at $c_j$ and are mapped by $f$ to a single ray;
\item the convex hulls 
in the closed unit disk of $\Theta_1(c_1),\ldots,\Theta_m(c_m)$ are pairwise disjoint;
\item for each critical point $c\in J_f$,
\[{\rm deg}(f|_c)-1=\sum_{c_j=c}\big(\#\Theta_j(c_j)-1\big). \]
\end{enumerate}
\end{definition}

\begin{remark}
\begin{enumerate}
\item Any \pf polynomial $f$ always has a weak critical marking. For example, let $c_1,\ldots,c_m$ be the pairwise different critical points of $f$ in $J_f$, and $U_1,\ldots,U_n$ the pairwise different critical Fatou components. We first construct $\Theta(U_1),\ldots,\Theta(U_n)$ as above. Pick a collection of angles  $\{\theta_1,\ldots,\theta_m\}$ such that the ray $\RRR(\theta_j)$ lands at $f(c_j)$ for each $j\in\{1,\ldots,m\}$. We then define each $\Theta_j(c_j)$ the set of arguments of the rays in $f^{-1}(\RRR(\theta_j))$ that land at $c_j$. It is easy to check that such defined $\Theta_1(c_1),\ldots,\Theta_m(c_m)$ satisfy conditions (2)--(5) in Definition \ref{weak-critical-marking}. Thus, together with the chosen $\Theta(U_1),\ldots,\Theta(U_n)$, one gets a weak critical marking $$\Th=\{\Theta_1(c_1),\ldots,\Theta_m(c_m);\Theta(U_1),\ldots,\Theta(U_n)\}.$$
Note that the weak critical markings of a \pf polynomial are not unique but there are finitely many choices.
\item Let $\Th$ be a weak critical marking of a \pf polynomial $f$ of the form  \eqref{eq:weak-critical-marking}. When the points $c_1,\ldots,c_m$ are pairwise different, the weak critical marking $\Th$ is further called  a \emph{critical marking} of $f$.  In this case,  the cardinality of each $\Theta_j(c_j)$ equals to ${\rm deg}(f|_{c_j})$.  We  remark that  our definition of  critical markings is  less restrictive than that of Poirier in \cite{Poi1}.
\end{enumerate}
\end{remark}

For example, we consider the \pf polynomial $f_c(z)=z^3+c$ with $c\approx 0.22036+1.18612 i$. The critical value $c$ receives two rays with arguments $11/72$ and $17/72$. Then, $$\Th:=\large\{\ \Theta_1(0):=\left\{11/216,83/216\right\},\Theta_2(0):=\left\{89/216,161/216\right\}\ \large\}$$
is a weak critical marking, but not a critical marking, of $f_c$, and
\[\Th:=\large\{\ \Theta_1(0):=\{11/216,83/216,155/216\}\ \large\}\]
is a critical marking of $f_c$ (see Figure \ref{portrait}).
\begin{figure}
\begin{tikzpicture}
\node at (0,0) {\includegraphics[width=6.5cm]{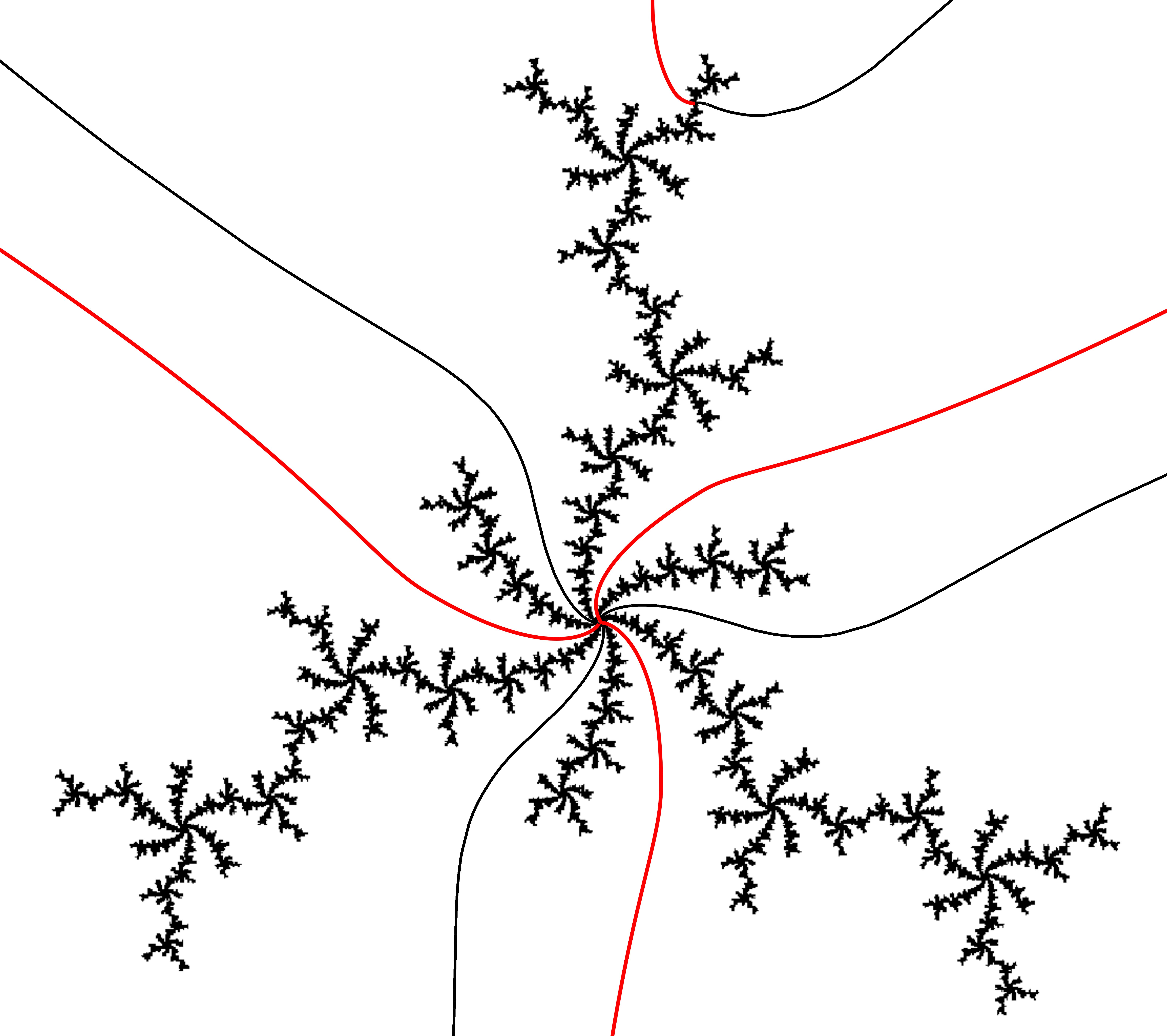}};
\node at (2.25,3){\footnotesize{$\frac{11}{72}$}};
\node at (0.25,3){\footnotesize{$\frac{17}{72}$}};
\node at (3.5,1.25){\footnotesize{$\frac{17}{216}$}};
\node at (3.5,0.25){\footnotesize{$\frac{11}{216}$}};
\node at (-3.5,2.5){\footnotesize{$\frac{83}{216}$}};
\node at(-3.5,1.5){\footnotesize{$\frac{89}{216}$}};
\node at(-1,-2.5){\footnotesize{$\frac{155}{216}$}};
\node at(0.5,-2.5){\footnotesize{$\frac{161}{216}$}};
\end{tikzpicture}
 \caption{The Julia set of $f_c(z)=z\mapsto z^3+0.22036+1.18612 i$.}\label{portrait}
\end{figure}

 In the following, we will use several notations as listed in the table below:
\begin{center}
\begin{tabular}{|c|c|c|}
\hline
&&\\[-3pt]
\small{$\mathbf{\Theta}_\FFF=\{\Theta(U_1),\ldots,\Theta(U_n)\}$} & \small{$\text{crit}(\mathbf{\Theta}_\FFF)=\cup_{k=1}^{n}\Theta(U_k)$} & \small{$\text{post}(\mathbf{\Theta}_\FFF)=\cup_{n\geq1}\tau^n(\text{crit}(\mathbf{\Theta}_\FFF))$}\\
&&\\[-3pt]
\hline &&\\[-3pt]
\small{$\mathbf{\Theta}_\JJJ=\{\Theta_1(c_1),\ldots,\Theta_m(c_m)\}$} & \small{$\text{crit}(\mathbf{\Theta}_\JJJ)=\cup_{j=1}^{m}\Theta_j(c_j)$} &\small{$\text{post}(\mathbf{\Theta}_\JJJ)=\cup_{n\geq1}\tau^n(\text{crit}(\mathbf{\Theta}_\JJJ))$}\\&&\\[-3pt]\hline
&&\\[-3pt]
\small{$\mathbf{\Theta}=\mathbf{\Theta}_\FFF\cup \mathbf{\Theta}_\JJJ$}&\small{$\text{crit}(\mathbf{\Theta})=\text{crit}(\mathbf{\Theta}_\FFF)\cup \text{crit}(\mathbf{\Theta}_\JJJ)$}& \small{$\text{post}(\mathbf{\Theta})=\text{post}(\mathbf{\Theta}_\FFF)\cup \text{post}(\mathbf{\Theta}_\JJJ)$}\\&&\\[-3pt]\hline
\end{tabular}
\end{center}

Set $\T:=\R/\Z$ and let $\tau:\T\to \T$ be the map defined by $\tau(\theta)=d\theta ({\rm mod}~\Z)$. By the construction, we immediately get the basic properties of $\Th$ (see also \cite[Section I.3]{Poi1}).

\begin{proposition}\label{define-portrait}
Enumerating the elements of $\Th$ by $\Theta_1,\ldots,\Theta_{m+n}$, then we have
\begin{enumerate}
\item each  $\tau(\Theta_i),i\in\{1,\ldots,m+n\}$,  is a singleton;
\item the convex hulls $hull(\Theta_i),hull(\Theta_j)$ in the closed unit disk of $\Theta_i,\Theta_j$ intersect at most at a point of $\T$ for any $i\not=j\in\{1,\ldots,m+n\}$;
\item each $\#\Theta_i\geq 2$, and $\sum_{i=1}^{m+n}(\#\Theta_i-1)=d-1$;
\end{enumerate}
\end{proposition}

In the construction of $\Theta(U)$, we obtain some internal rays of $U$ associated with $\Theta(U)$. These internal rays will play an important role in the following discussion.

Remember that $\g(\theta)$ denotes the landing point of $\RRR(\theta)$, and that $r_U(\theta)$ denotes the internal ray of $U$  landing at $\g(\theta)$ if $\g(\theta)\in\partial U$.

\begin{definition}[critical/postcritical internal rays]\label{critical-internal-ray}
An internal ray is called a \emph{critical internal ray} (relatively to $\Th$) if it can be represented as $r_U(\theta)$ for a critical Fatou component $U$ and $\theta\in\Theta(U)$; and called a \emph{postcritical internal ray} (relatively to $\Th$) if it is an iterated image by $f$ of a critical internal ray.
\end{definition}

By the construction of the weak critical marking $\Th$, we have the following result.

\begin{lemma}\label{basic-property}
\begin{enumerate}
\item For each $\theta\in{\rm crit}(\Th_\FFF)/{\rm post}(\Th_\FFF)$, there is a critical/postcritical Fatou component $U$ such that $r_U(\theta)$ is a critical/postcritical internal ray.
\item If $r_U(\theta)$ is a critical/postcritical internal ray, then $f(r_U(\theta))=r_{f(U)}(\tau(\theta))$ is a postcritical internal ray.
\item Each critical/postcritical internal ray is eventually periodic under the iterations of $f$.
\item The closure of any critical Fatou component $U$ contains no other critical/postcritical internal rays except those $r_U(\theta)$ with $\theta\in\Theta(U)$.
\item The closure of each Fatou component contains at most one periodic critical/postcritical internal ray.
\end{enumerate}
\end{lemma}

\subsection{The critical portrait}\label{critical} 

 The combinatorial information given in Proposition \ref{define-portrait} for weak critical markings of \pf polynomials may be presented in an abstract way, thus giving rise to the concept of  \emph{critical portrait}.

 Denote by $\D$ the unit disk. By abuse of notation, we will identify any point of $\partial\D$ with its argument in $\T$.
Then all angles in the circle are considered to be  mod 1.
 A \emph{leaf} is either one point in $\T$ (trivial) or the closure in $\overline \D$ of a hyperbolic chord (non-trivial).  For $x,y\in \T$, we use $\ov{xy}$ to denote the leaf joining $e^{2\pi i x}$ and $e^{2\pi i y}$. For any set $S\subset \T$, we denote by $hull(S)$ the hyperbolic convex hull in $\ov{\D}$ of $S$.

  \begin{definition}[critical portrait]\label{formal-critical-portrati}
  A degree $d$ \emph{critical portrait} is a finite collection of finite subsets of the unit circle  $ \mathbf{\Theta}=\{ \Theta_1, \cdots,  \Theta_s\}$
 such that  the properties (1), (2) and (3) in Proposition \ref{define-portrait} simultaneously hold.



 \end{definition}

\begin{figure}
\begin{center}
 \includegraphics[scale=0.32]{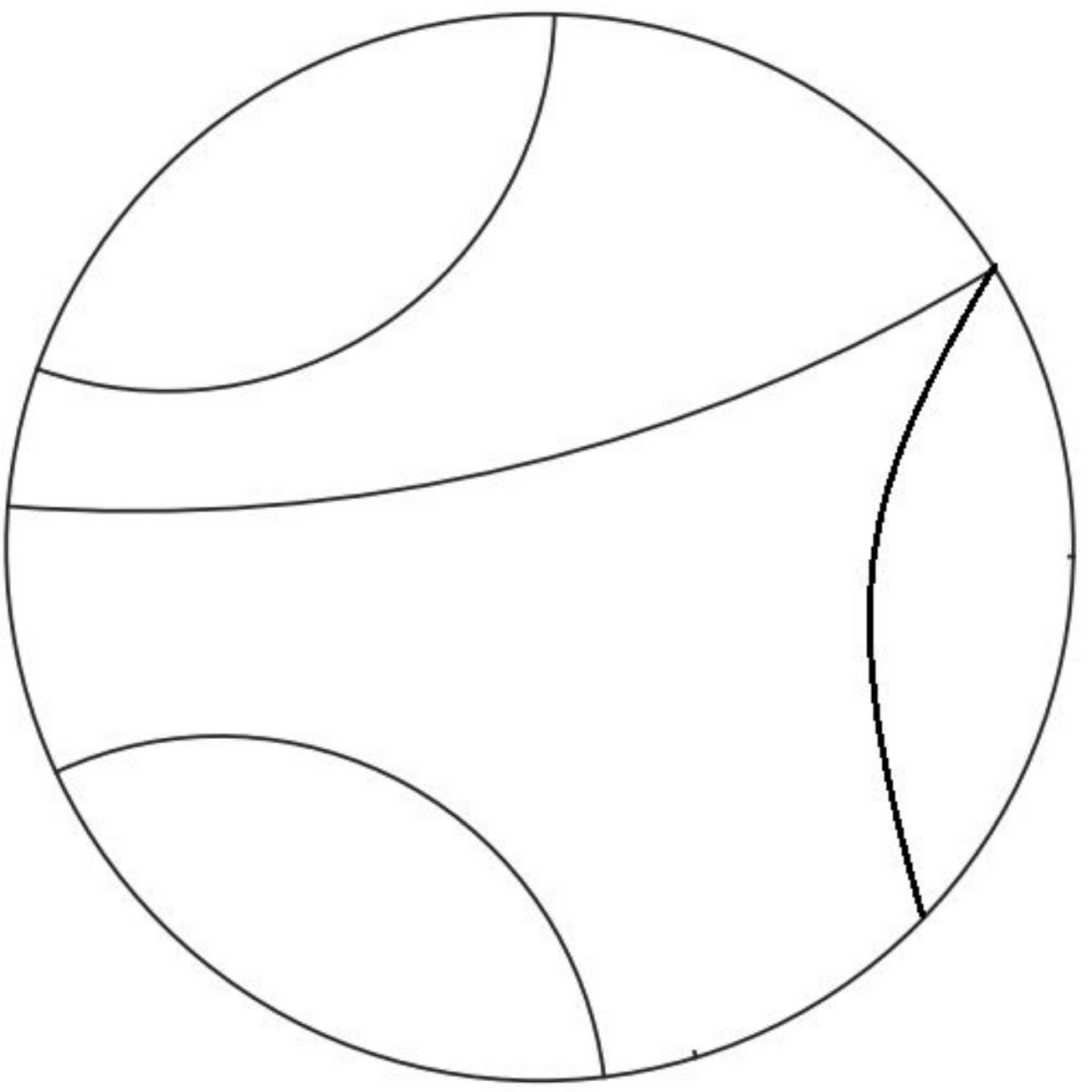}  \quad \includegraphics[scale=0.45]{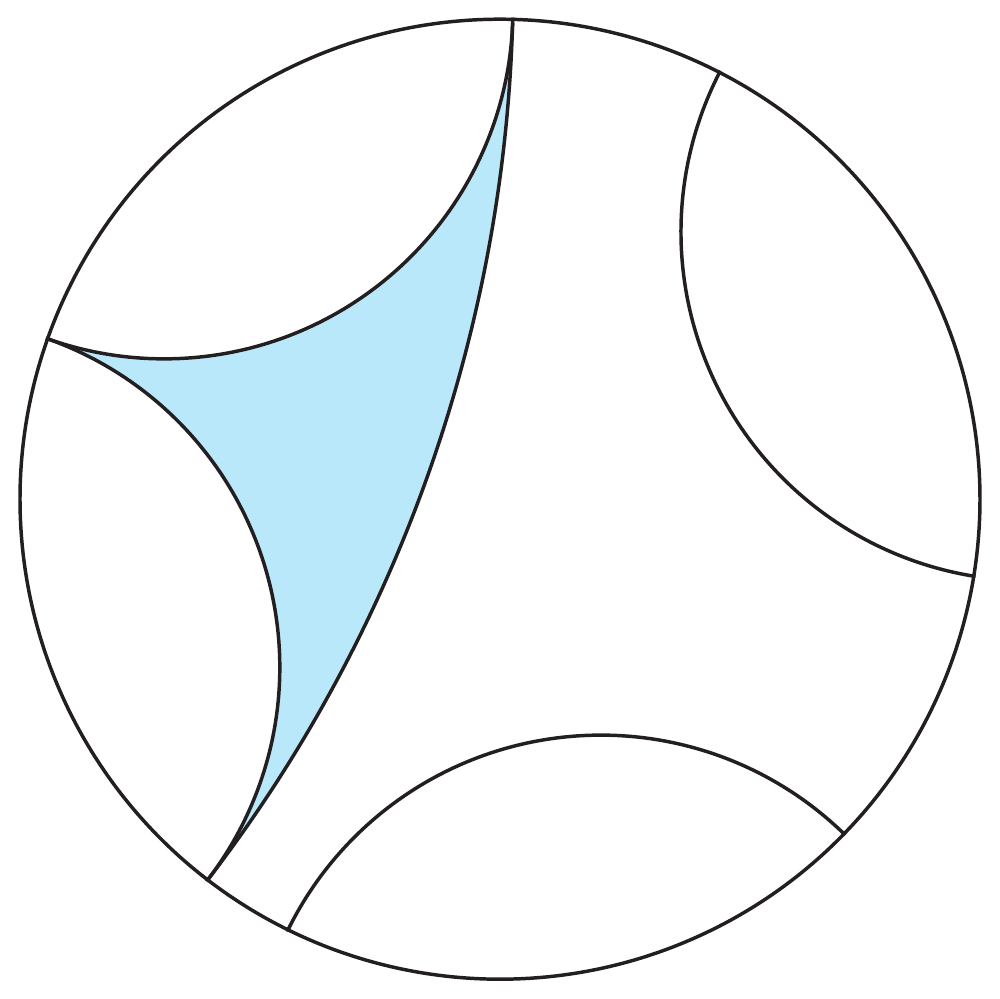}
  \caption{Two  critical portraits of degree $5$}\label{critical-portrait-5}
  \end{center}
\end{figure}

Thus,  any weak critical marking of a \pf polynomial is a critical portrait.
By identifying each $\Theta_i\in \Th$ with its convex hull, a critical portrait can be equivalently considered as  a collection of  leaves and polygons
pairwise disjoint in $\D$, each of whose vertices are identified under $z\mapsto z^d$, with total criticality $d-1$ (see Figure \ref{critical-portrait-5}).

 \begin{lemma}\label{equivalent}
 If $\mathbf{\Theta}=\{\Theta_1,\ldots,\Theta_s\}$ is a critical portrait of degree $d$, then $$\D\setminus\big(\cup_{i=1}^s hull(\Theta_i)\big)$$ has $d$ connected components, and
each one takes a total arc length $1/d$  on the unit circle.\vspace{-3pt}
\end{lemma}
\begin{proof}
The set $hull(\Theta_1)$ chops $\D$ into $\#\Theta_1$ regions. One of them contains $hull(\Theta_2)$,
and is chopped by $hull(\Theta_2)$ into $\#\Theta_2$ regions. As a result, $hull(\Theta_1)$ and $hull(\Theta_2)$ together chop $\D$ into $\#\Theta_1 + (\#\Theta_2 - 1)$ regions. If we chop
$hull(\Theta_3), \ldots, hull(\Theta_s)$ consecutively, at the end we get that the union of convex hulls of $\Theta_i$'s chop $\D$ into
 $\#\Theta_1 + \sum_{i = 2}^s (\#\Theta_i - 1)$ regions, and the condition (3) says this number is exactly $d$.

 On the other hand, each of these $d$ regions touches the boundary circle in arcs whose total length is a multiple of $1/d$, since we know $\Theta_i$ decomposes $\T$ into arcs each of which has length a multiple of $1/d$. Now we have the union of $d$  arcs each of which has total length a multiple of $1/d$,
so each one must be of length exactly $1/d$ in order to sum up to $1$.
\end{proof}

Let $\mathbf{\Theta}=\{\Theta_1,\ldots,\Theta_s\}$ be a critical portrait of degree $d\geq2$.

\begin{definition}[unlinked equivalent on $\T$] Two points $x,y\in\T\setminus\cup_{i=1}^s\Theta_i$ are called \emph{unlinked on $\T$} relatively to $\mathbf{\Theta}$ if they belong to a common component of $\T\setminus \Theta_i$ for all  $\Theta_i\in \mathbf{\Theta}$.
\end{definition}
This unlinked relation is easily checked to be an equivalence relation. Together with Lemma \ref{equivalent}, we immediately deduce  the following result.
\begin{proposition}\label{common}
\begin{enumerate}
\item Two points $x,y\in\T\setminus\cup_{i=1}^s\Theta_i$ are unlinked on $\T$ if and only if they belong to a common complementary component of $ \ov{\D}\setminus\big(\cup_{i=1}^s hull( \Theta_i)\big)$.
\item There are $d$ unlinked equivalence classes on $\T$, denoted by $I_1,\ldots, I_d$. For each $I_k$ and any angles $x,y\in \ov{I_k}$, $\tau(x)=\tau(y)$ if and only if there exist $\Theta_{i_1},\ldots,\Theta_{i_t}\in\Th$ with $t\geq1$ such that $x\in\Theta_{i_1},y\in \Theta_{i_t}$ and $\Theta_{i_j}\cap\Theta_{i_{j+1}}\not=\emptyset$ for all $j=1,\ldots,t-1$.
\end{enumerate}
\end{proposition}

\subsection{Partitions in the dynamical plane induced by weak critical markings}\label{partition-in-f}

In this part, we fix a weak critical marking
 \[\mathbf{\Theta}=\{\Theta_1(c_1),\ldots, \Theta_m(c_m);\ \Theta(U_1),\ldots,\Theta(U_n)\}\triangleq\{\Theta_1,\ldots,\Theta_{m+n}\}\]
 of a \pf polynomial $f$ of degree $d$. As shown in the last section, it induces a partition $I_1,\ldots, I_d$ on the unit circle. We will introduce in this section a corresponding  partition on the dynamical plane of $f$, see also \cite[Section II.2]{Poi1} and \cite[Section 5]{Z}.

Firstly, for each $c_j$ with $j\in\{1,\ldots,m\}$ and each critical Fatou component $U$ of $f$, we define the sets $\RRR_j(c_j)$ and $\RRR(U)$ in the dynamical plane corresponding to $\Theta_j(c_j)$ and $\Theta(U)$ respectively, such that $\RRR_j(c_j)$ is the union of the external rays with angles in $\Theta_j(c_j)$ and the point $c_j$, i.e.,
\[\RRR_j(c_j):=\big(\cup_{\theta\in\Theta_j(c_j)}\RRR(\theta)\big)\cup\{c_j\},\]
and $\RRR(U)$ is the union of extended rays at $U$ with angles in $\Theta(U)$, i.e.,
\[\RRR(U):=\cup_{\theta\in\Theta(U)}\EEE_U(\theta).\]
According to the construction of $\Theta_j(c_j)$ and $\Theta(U)$ in Section \ref{section-critical-portrait}, we get the following properties of $\RRR_j(c_j)$ and $\RRR(U)$.
\begin{lemma}\label{property-R(c)}
\begin{enumerate}
\item Each of $\RRR_j(c_j)$ and $\RRR(U)$ is star-like with a critical point in the center. Moreover, $\RRR_j(c_j)\cap K_f=\{c_j\}$ and $\RRR(U)\cap K_f$ consists of critical internal rays in $U$.
\item the intersection of $\RRR_j(c_j)$ and $\RRR_k(c_k)$ with $j\not=k\in\{1,\ldots,m\}$ is equal to $\{c_j\}$ if $c_j=c_k$, and an empty set otherwise.
\item If $\RRR_j(c_j)\cap \RRR(U)\not=\emptyset$, then $c_j\in\partial U$, and the intersection is either $\{c_j\}$ or the union of $\{c_j\}$ and an external ray $\RRR(\theta)$ landing at $c_j$. The latter case happens if and only if $\Theta_j(c_j)\cap \Theta(U)=\{\theta\}$.
\item If $\RRR(U)\cap \RRR(U')\not=\emptyset$ for distinct critical Fatou components $U,U'$, then the intersection is either a point $\{p\}:=\partial U\cap \partial U'$, or the union of $\{p\}$ and an external ray $\RRR(\theta)$ landing at $p$. The latter case happens if and only if $\Theta(U)\cap\Theta(U')=\{\theta\}$.
\end{enumerate}
\end{lemma}

For each $\Theta\in \Th$, we simply denote
\begin{equation}\label{eq:1}
\RRR(\Theta):=\left\{
                  \begin{array}{ll}
                    \RRR_j(c_j), & \hbox{if $\Theta=\Theta_j(c_j)$ with $j\in\{1,\ldots,m\}$;} \\[5pt]
                    \RRR(U), & \hbox{if $\Theta=\Theta(U)$ with $U$ a critical Fatou component.}
                  \end{array}
                \right.
\end{equation}


\begin{definition}[unlinked equivalence in the $f$-plane]\label{unlink-f}
We say that two points $z_1,z_2$ of $\C\setminus \bigcup_{i=1}^{m+n}\RRR(\Theta_i)$ are \emph{unlinked equivalent in the $f$-plane} if they belong to a common connected component of $\C\setminus \RRR(\Theta_i)$  for every $\Theta_i\in\Th$.
\end{definition}

Looking at the circle at infinity we immediately derive that 
two external rays $\RRR(\theta)$ and $\RRR(\eta)$  are in a common unlinked equivalence class in the $f$-plane  if and only if $\theta$ and $\eta$ are in a common unlinked equivalence class on $\T$. This provides a canonical correspondence between the unlinked equivalence classes on $\T$ and the ones in the $f$-plane. We then
denote as  $$ V_1,\ldots, V_d$$ the unlinked equivalence classes in the $f$-plane with $ V_k$ corresponding to $I_k, k\in\{1,\ldots,d\}$ (see Figure \ref{partition}). By Lemma \ref{property-R(c)} and Definition \ref{unlink-f}, the boundary of each $V_k$ consists of internal/external rays, and $\partial V_k\cap K_f$ consists of either critical points in $J_f$ or critical internal rays.
 \begin{figure}
\begin{tikzpicture}
\node at (-4,0) {\includegraphics[width=6cm]{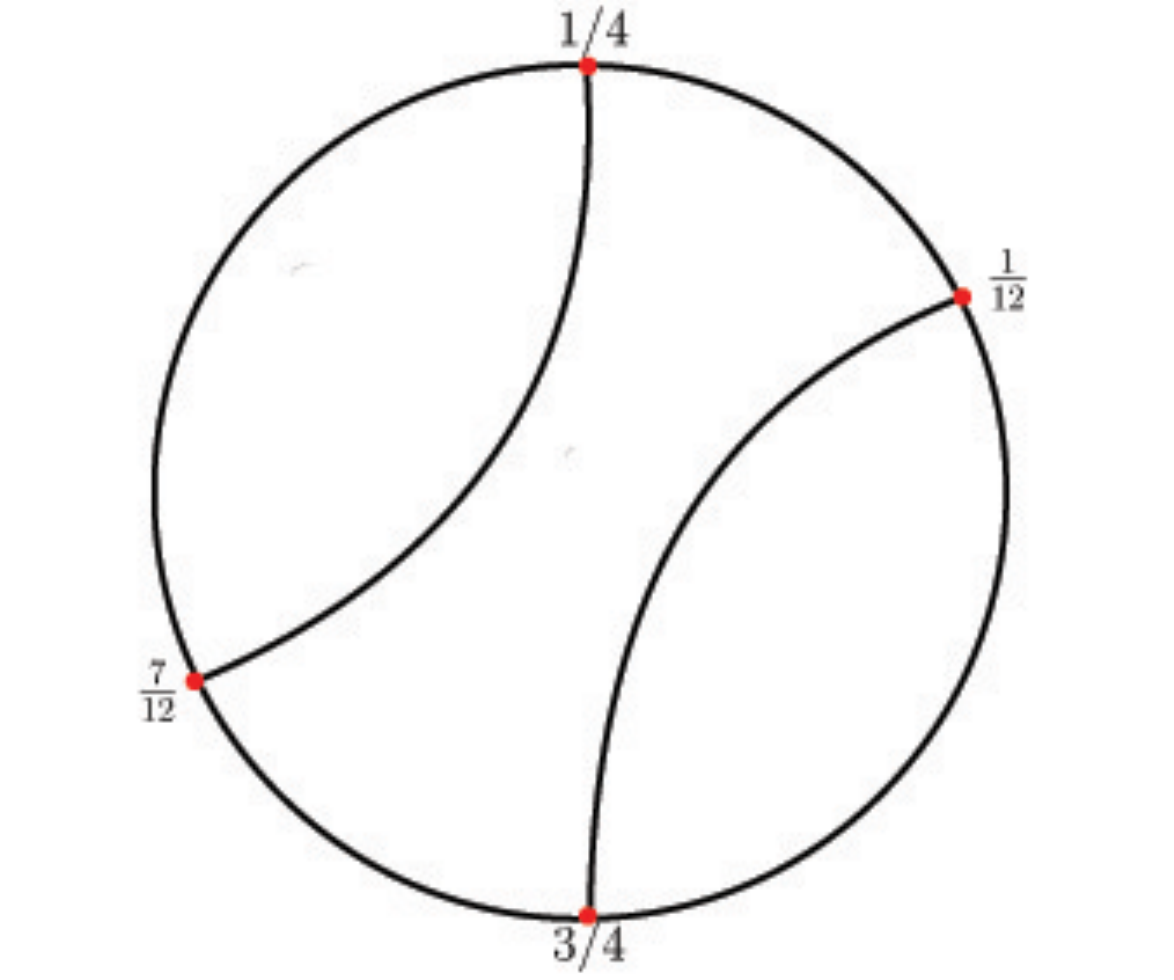}};
\node at (3.5,0) {\resizebox{8cm}{4.5cm}{\includegraphics{partition-1.pdf}}};
\node at (3.1,2){\footnotesize{$\RRR(\frac{1}{4})$}};
\node at (7.2,2){\footnotesize{$\RRR(\frac{1}{12})$}};
\node at (0.85,-2){\footnotesize{$\RRR(\frac{7}{12})$}};
\node at (3.85,-2){\footnotesize{$\RRR(\frac{3}{4})$}};
\node at (2,0){\footnotesize{$U_A$}};
\node at(5,0){\footnotesize{$U_B$}};
\node at(0,1.25){\footnotesize{$V_1$}};
\node at(2.5,-2){\footnotesize{$V_2$}};
\node at(4.5,2){\footnotesize{$V_2$}};
\node at(6.5,-1.5){\footnotesize{$V_3$}};
\node at(-5.5,1){\footnotesize{$I_1$}};
\node at (-4,0){\footnotesize{$I_2$}};
\node at (-2.5,-1){\footnotesize{$I_3$}};
\end{tikzpicture}
 \caption{This example comes from \cite[Section I.4.4]{Poi1}. The cubic polynomial $f(z)=z^3-\frac{3}{2}z$ has a critical portrait $\Th=\big\{\Theta(U_A)=\{\frac{1}{4},\frac{7}{12}\},\Theta(U_B)=\{\frac{3}{4},\frac{1}{12}\}\big\}$. It induces a partition on the unit circle (the left figure), and the corresponding partition in the $f$-plane.}\label{partition}
\end{figure}

\begin{lemma}\label{injective-f}
Fix any unlinked equivalence class $ V$ in the $f$-plane and denote its closure by $\ov{ V}$. Then:
\begin{enumerate}
\item For two distinct points $z_1,z_2\in \ov{ V}\cap K_f$, the regulated arc $[z_1,z_2]\subset \ov{ V}$.
\item The polynomial $f$ maps $ V$ bijectively onto $\C\setminus f(\partial V)$.
\item If $f(z)=f(w)$ for  $z\not=w\in \ov{ V}\cap K_f$, then the regulated arc $[z,w]$ belongs to $\partial  V$.
\item For any $z,w\in \ov{ V}\cap K_f$, the restriction of $f$ on $[z,w]\setminus\partial  V$ is injective, and its image is contained in a regulated arc of $f([z,w])$.
\end{enumerate}
\end{lemma}
\begin{proof}
\begin{enumerate}
\item  For each $i\in\{1,\ldots,m+n\}$, denote by $Y_i$ the connected component of $\C\setminus\RRR(\Theta_i)$ which contains $ V$. By Definition \ref{unlink-f} we get $ V=\cap_{i=1}^{m+n} Y_i$. By Lemma \ref{property-R(c)}, we further get $\ov{ V}=\cap_{i=1}^{m+n} \ov{Y_i}$. It follows that $z,w$ belong to each $\ov{Y_i}$. Note that the intersection of $\partial Y_i$ and $K_f$ is either a point in $J_f$ or two internal rays of a  Fatou component, then $[z,w]$ is contained in every $\ov{Y_i}$.  Consequently, we have $[z,w]\subset \cap_{i=1}^{m+n}\ov{Y_i}=\ov{ V}$.

\item Note that the boundary of $V$ consists of external/internal rays. It follows that the image $f(\partial  V)$ is the union of finitely many external or internal rays. Let $I$ denote the unlinked equivalence class on $\T$ that corresponds to $ V$. From the correspondence between $I$ and $ V$, and Proposition \ref{common}.(2), we get that $f$ maps $ V\cap U(\infty)$ bijectively onto the region $U(\infty)\setminus f(\partial  V)$.

We claim that $f( V)\subset \C\setminus f(\partial  V)$. On the contrary, let $z\in  V\cap K_f$ such that $f(z)\in f(\partial  V)$. If $z\in U$, a Fatou component, denote by $c$ the center of $U$. Since $\partial V\cap K_f$ consists of internal rays and $z\in V$, then, if $z=c$, the entire $U$ is contained in $V$; otherwise, there is a unique internal ray of $U$ passing through $z$, which is contained in $V\cup\{c\}$. Anyway, we can pick an internal ray $r_z$ of $U$ with $r_z\subset V\cup\{c\}$. So, by choosing the landing point of $r_z$ if necessary, we can assume $z\in J_f$. In this case, there is an external ray contained in $f(\partial  V)$ landing at $f(z)$, and its lift based at $z$ is also an external ray. As $z\in V$,  all external rays landing at $z$ belong to $ V\cap U(\infty)$. It contradicts to the fact that $f( V\cap U(\infty))=U(\infty)\setminus f(\partial  V)$. The claim is then proven.

Now, we fix a point $w_W\in U(\infty)$ in each component $W$ of $\C\setminus f(\partial  V)$ and denote by $z_W$ the unique preimage of $w_W$ in $ V\cap U(\infty)$. Given any connected component $W$ of $\C\setminus f(\partial  V)$, and any point $w\in W$, we pick an arc $\g_w\subset W$ joining $w$ and $w_W$. We assert that $\g_w$ has a unique lift in $ V$ starting from $z_W$.  Firstly, let $\G$ denote the component of $f^{-1}(\g_w)$ that contains $z_W$. Then $\G\subset V$ since $z_W\in V$ and $\g_w \cap f(\partial  V)=\emptyset$. Note also that $V$ contains no critical points of $f$, so $f:\G\to\g_w$ is a covering, and hence a homeomorphism. We denote by $h(w)$ the end point of this lift. Since $W$ is simply-connected,  the point $h(w)$ is independent on the choice of $\g_w$. We then get a map $h:\C\setminus f(\partial  V)\to  V$. This map is easily checked to be the inverse  of $f: V\to \C\setminus f(\partial V)$.

\item Let $z,w$ be two distinct points in $\ov{ V}\cap K_f$ with $f(z)=f(w)$. We first claim that both $z$ and $w$ belong to $\partial  V$. Otherwise we can find two points $z'$ and $w'$ in $ V$ near $z$ and $w$, respectively, such that $f(z')=f(w')$ and this contradicts assentation  (2).

  Assume now that $z,w\in \partial  V$ and let $[z,w]$ be the regulated arc joining $z,w$. By point (1), $[z,w]\subset \ov{ V}$. We just need to show that no points of $[z,w]$ are in $ V$. On the contrary, suppose that $x\in  V\cap [z,w]$. Note that the image of $[z,w]$ is a tree, so the fact $f(z)=f(w)$ implies that each point in $f([z,w])$ has at least two preimages in $[z,w]$ by $f$. Then there is  a point $y\in[z,w]$ such that $f(x)=f(y)$. It contradicts the claim above.

 \item  The former assertion is a direct consequence of (3).
Note that each connected component of $[z,w]\cap \partial  V$ is either one point or a closed segment, so we
 denote by
 $$A_1:=[z_1,w_1],\ldots,A_m:=[z_m,w_m],\ m\geq0,$$
 the closures of the connected component of $[z,w]\setminus\partial  V$ such that $A_i$ separates $A_{i-1}$ and $A_{i+1}$ for each $i\in [2,m-1]$. Here $m=0$ represents the case that $[z,w]\subset \partial  V$. Obviously, if $m\geq1$, each $A_i$ is a closed arc. For each $i\in[1,m-1]$, we denote by  $B_i$ the connected component of $[z,w]\cap\partial  V$  between $A_i$ and $A_{i+1}$ (see the left one in Figure \ref{map}).
\begin{figure}
 \begin{center}
 \includegraphics[width=15cm]{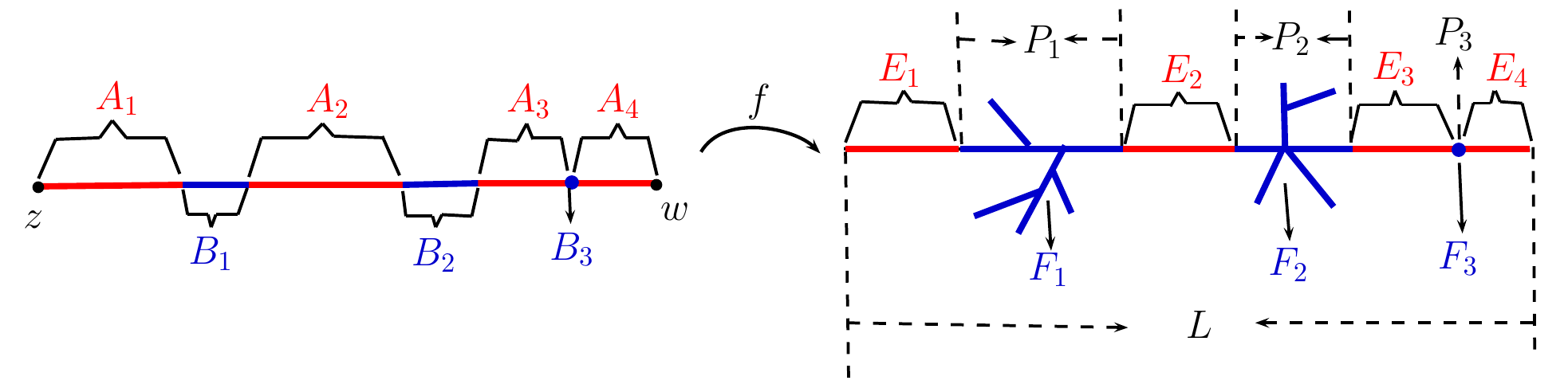}
  \caption{The decomposition of $[z,w]$ and its image in (4)}\label{map}
  \end{center}
\end{figure}

For each $i\in\{1,\ldots,m\}$, set $E_i:=f(A_i)$. According to (3), we get that $E_i=[f(z_i),f(w_i)]$ and their interior are pairwise disjoint. We also set $F_i:=f(B_i),i\in\{1,\ldots,m\}$. Clearly, each $F_i$ is a tree or point containing $f(w_i)$ and $f(z_{i+1})$ (see the right one in Figure \ref{map}).

We claim that except  these two points, the set $F_i$ is disjoint from $\bigcup_{j=1}^{m}E_j$. To see this, let $p\in B_i$ with $f(p)\in E_j$. Since $E_j=f(A_j)$, there exists a point $q\in A_j$ such that $f(q)=f(p)$. By (3) and the discussion above, the point $q$ also belongs to $B_i$. It implies that either $j=i$ and $q=w_i$ or $j=i+1$ and $q=z_{i+1}$. Then the claim is proven.

Consequently, we denote by $P_j:=[f(w_j),f(z_{j+1})]$ an arc joining $E_j$ and $E_{j+1}$ for each $j\in\{1,\ldots,m-1\}$. Then $P_j\subset F_j$ and $P_j\cap \big(\cup_{i=1}^{m}E_i\big)=\{f(w_j),f(z_{j+1})\}$. Therefore $$L:=\big(\cup_{i=1}^mE_i\big)\cup\big(\cup_{j=1}^{m-1}P_j\big)$$ is a regulated arc and contained in $f([z,w])$ (see Figure \ref{map}).
\end{enumerate}
\end{proof}

\section{The core entropy and Thurston's entropy algorithm}\label{section-algorithm}

The purpose of this section is to define the notions in the following diagram and provide the background for proving Theorem \ref{Thurston-algorithm}.

 \[
 \begin{tikzpicture}
   \matrix[row sep=0.8cm,column sep=4cm] {
     \node (Gammai) {$ \text{critical portrait  }\mathbf{\Theta}$}; &
       \node (Gamma) {$\log\rho(\mathbf{\Theta})$}; \\
     \node (S2i) {p.f polynomial $f$}; &
       \node (S2) {$h(H_f, f)$.}; \\
   };
   \draw[double equal sign distance] (Gamma) to node[auto=left,cdlabel] {\text{ equality by Thm. \ref{Thurston-algorithm}} } (S2);
   \draw[->] (S2i) to node[auto=right,cdlabel] {\text{the core entropy of $f$}} (S2);
   \draw[->] (Gammai) to node[auto=left,cdlabel] {\text{Thurston's  entropy algorithm}} (Gamma);
   \draw[->]  (S2i) to node[auto=right,cdlabel] {} (Gammai);
 \end{tikzpicture}
 \]



\subsection{Separation sets}


 In order to introduce Thurston's entropy algorithm, we need to consider the position of two points of $\T$ relatively to a critical portrait.


Let $\Th=\{\Theta_1,\ldots,\Theta_s\}$ be a critical portrait.
 Given two angles $x,y\in \T$ (not necessarily different) and an element $\Theta\in\Th$, we say that the leaf $\ov{xy}$ \emph{crosses} $hull(\Theta)$ if $x,y\not\in \Theta$ and $\ov{xy}\cap hull(\Theta)\not=\emptyset$. In this case, the angles $x, y$ are said to be \emph{separated} by $\Theta$.

\begin{definition}[separation set]\label{separation-set}
Given an ordered pair of angles $x,y$ in $\T$, we say that its \emph{separation set} (relatively to $\Th$) is
$(k_1, \dots, k_p)$ if the leaf $\ov{xy}$ successively crosses $hull(\Theta_{k_1}),\ldots hull(\Theta_{k_p})$ from $x$ to $y$, where $\Theta_{k_1},\ldots,\Theta_{k_p}\in \Th$, and no other elements of $\Th$ separate $x,y$.
The angles $x$ and $y$ are called \emph{non-separated} by $\Th$ if its separation set is empty.
\end{definition}

\begin{lemma}\label{separation}
Two angles  $x,y\in\T$ are non-separated by $\Th$ if and only if there exist an unlinked equivalence class $I$ on $\T$ and two elements $\Theta_i,\Theta_j\in\Th$ intersecting $\ov{I}$ such that $x,y\in I\cup \Theta_i\cup \Theta_j$.
\end{lemma}
\begin{proof}
It follows directly from the definition of non-separated.
\end{proof}

By this lemma, it is easy to see that
\begin{corollary}\label{non-separated}
If $x,y\in \T$ has the separation set $(k_1,\ldots,k_p)\not=\emptyset$, then, by setting $\theta_0:=x,\theta_{p+1}:=y$ and the choice of any angle $\theta_i\in\Theta_{k_i}$, each pair of angles $\theta_i,\theta_{i+1}$ for $i\in\{0,\ldots,p\}$ is non-separated by $\Th$.
\end{corollary}

\subsection{Thurston's entropy algorithm }\label{algorithm2}


A critical portrait is said to be  \emph{rational} if each of its elements contains only rational angles. For instance, any weak critical marking of a \pf polynomial is rational.
Throughout this subsection, we fix a rational critical portrait
\[\Th=\{\Theta_1,\ldots,\Theta_s\}.\]

Recall that $\text{crit}(\Th)=\cup_{i=1}^s\Theta_i$ and $\text{post}(\Th)=\cup_{n\geq1}\tau^n(\text{crit}(\Th))$.
 We define a set $\SSS=\SSS_\Th$, consisting of all unordered pairs
$\{x,y\}$ with $x\not=y\in \text{post}(\Th)$ if $\#{\rm post}(\Th)\geq 2$, and consisting of only $\{x,x\}$ if ${\rm post}(\Th)=\{x\}$. Note that in the latter case $x$ is fixed by $\tau$. Since $\Th$ is rational, the sets $\text{post}(\Th)$ and $\SSS$ are finite and non-empty.
The following is the procedure of Thurston's entropy algorithm acting on the rational critical portrait ${\Th}$.
\begin{enumerate}
\item Let $\Sigma$ be the abstract linear space over $\R$ generated by the elements of $\SSS$.

\item  Define a linear map $\mathcal{A}: \Sigma\longrightarrow \Sigma$ such that for any basis vector $\{x,y\}\in \SSS$,
\begin{enumerate}
\item $\mathcal{A}(\{x,y\} )=\{\tau(x),\tau(y)\}$ if $x,y$ are non-separated by ${\Th}$; and
\item $\mathcal{A}(\{x,y\} )=\sum\limits_{i=0}^{p}\mathcal{A}(\{\theta_i, \theta_{i+1}\})$, where $\theta_0:=x,\theta_{p+1}:=y$ and $\theta_i\in\Theta_{k_i}\in\Th$,
if the ordered pair $(x,y)$ has the separation set $(k_1,\ldots,k_p)\not=\emptyset$.
 \end{enumerate}

\item  Denote by $A=A_{\Th}$ the matrix of $\mathcal{A}$ in the basis  $\SSS$. It is a non-negative matrix. Compute its leading non-negative eigenvalue $\rho({\Th})$ (such an eigenvalue exists by the Perron-Frobenius theorem).
\end{enumerate}

\begin{definition}\label{alto} {\rm (Thurston's entropy algorithm)} Take a rational critical portrait $\Th$ as input and  $ \log \rho(\Th)$ as output. (It is easy to see that $A$ is not nilpotent therefore $\rho(\Th)\ge 1$).\end{definition}

\begin{remark}
\begin{enumerate}
\item By Corollary \ref{non-separated}, each $\AAA(\{\theta_i,\theta_{i+1}\})$ in (2).(b) has been defined in (2).(a).
\item By Lemma \ref{separation} and the definition in (2).(a), each $\AAA(\{\theta_i,\theta_{i+1}\})$, and hence $\AAA(\{x,y\})$, in (2).(b) is independent on the choice of $\theta_i\in\Theta_{k_i}$.
\end{enumerate}
\end{remark}

\begin{example}
Here and in all the examples $d=3$.
Let $\Th=\{\ \{0,1/3\},\{7/15,4/5\}\ \}$. Then the set ${\rm post}(\Th)=\{0,1/5,2/5,3/5,4/5\}$
 \begin{center}
  \includegraphics[scale=0.58]{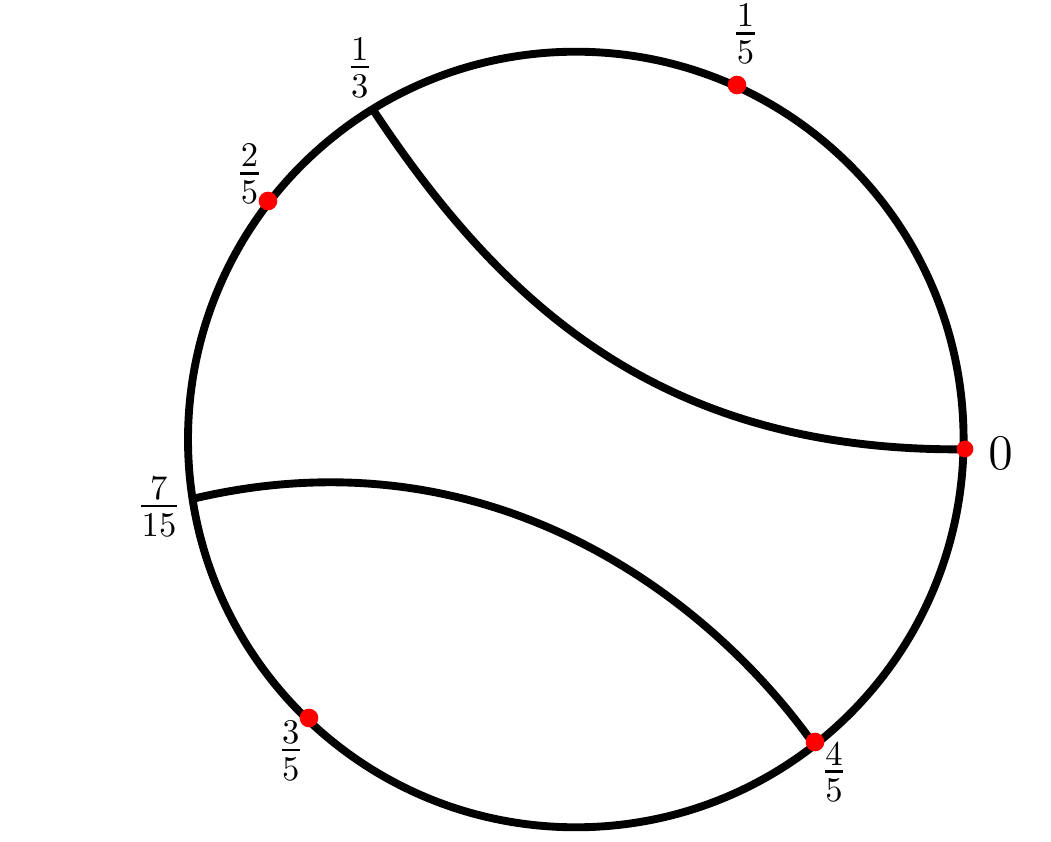}
  \end{center}
gives rise to an abstract linear space $\Sigma$ with  basis:
$$\SSS= \Bigl\{\{0,\frac{1}{5}\},\{0,\frac{2}{5}\},\{0,\frac{3}{5}\},\{0,\frac{4}{5}\},\{\frac{1}{5},\frac{2}{5}\},
\{\frac{1}{5},\frac{3}{5}\},\mystrut\{\frac{1}{5},\frac{4}{5}\},\{\frac{2}{5},\frac{3}{5}\},\{\frac{2}{5},\frac{4}{5}\},\{\frac{3}{5},\frac{4}{5}\}\Bigr\}$$
 The linear map $\AAA$ acts on the basis vectors as follows:
$$ \begin{array}{l}
\tiny{\Big\{0,\dfrac{1}{5}\Big\}\rightarrow\Big\{0,\dfrac{3}{5}\Big\},\quad \Big\{0,\dfrac{2}{5}\Big\}\rightarrow\Big\{0,\dfrac{1}{5}\Big\}, \quad \Big\{0,\dfrac{3}{5}\Big\}\rightarrow\Big\{0,\dfrac{2}{5}\Big\}+\Big\{\dfrac{2}{5},\dfrac{4}{5}\Big\},
\quad \Big\{0,\dfrac{4}{5}\Big\}\rightarrow\Big\{0,\dfrac{2}{5}\Big\}},\\[15pt]
\tiny{ \Big\{\dfrac{1}{5},\dfrac{2}{5}\Big\}\rightarrow\Big\{0,\dfrac{3}{5}\Big\}+\Big\{0,\dfrac{1}{5}\Big\}, \quad  \Big\{\dfrac{1}{5},\dfrac{3}{5}\Big\}\rightarrow\Big\{0,\dfrac{3}{5}\Big\}+\Big\{0,\dfrac{2}{5}\Big\}+\Big\{\dfrac{2}{5},\dfrac{4}{5}\Big\}, \quad \Big\{\dfrac{1}{5},\dfrac{4}{5}\Big\}\rightarrow\Big\{0,\dfrac{3}{5}\Big\}+\Big\{0,\dfrac{2}{5}\Big\}},\\[15pt]
\tiny{\Big\{\dfrac{2}{5},\dfrac{3}{5}\Big\}\rightarrow\Big\{\dfrac{1}{5},\dfrac{2}{5}\Big\}+\Big\{\dfrac{2}{5},\dfrac{4}{5}\Big\},\quad \Big\{\dfrac{2}{5},\dfrac{4}{5}\Big\}\rightarrow\Big\{\dfrac{1}{5},\dfrac{2}{5}\Big\}\quad  \Big\{\dfrac{3}{5},\dfrac{4}{5}\Big\}\rightarrow\Big\{\dfrac{4}{5},\dfrac{2}{5}\Big\}}. \end{array}$$
We compute $\log\rho(\Th)=1.395$.
\end{example}


\subsection{Relating Thurston's entropy algorithm to polynomials}\label{T-A}


 In this part, we will give an intuitive feeling of the relation between the output in the algorithm above and the core entropy of \pf polynomials, and leave the detailed proof to the next section.

Let $f$ be a \pf polynomial of degree $d\geq2$ and $\Th$  a weak critical marking of $f$ as constructed in Section \ref{section-critical-portrait}. It is known that the core entropy of $f$, i.e., the topological entropy of $f$ on $H_f$, is equal to  $\log \rho$ with $\rho$ being the leading eigenvalue of the incidence matrix  $D_{(H_f,f)}$ when considering  the Markov partition of $H_f$ by its edges. Instead, if one looks at the arcs of $H_f$ between any two postcritical points rather than the edges, the action of $f$ on $H_f$ induces another incidence matrix $A_f$ which turns out to have  the same leading eigenvalue $\rho$.

The advantage of this approach lies in the fact that each postcritical point of $f$ corresponds an angle in ${\rm post}(\Th)$ so that any arc in $H_f$ between postcritical points can be combinatorially represented by an angle pair (not necessarily unique). Thus, intuitively, the action of $f$ on these arcs induces a linear map on the space generated by the angle pairs with angles in ${\rm post}(\Th)$, which is the map $\AAA$ in the algorithm. Note that the matrix $A_\Th$ in the algorithm is in general  larger than $A_f$ because one postcritical point of $f$ usually corresponds to several angles in ${\rm post}(\Th)$. What we need to do in this paper is to show that these two matrices  have the same leading eigenvalue.

In the quadratic case, a complete proof of Theorem \ref{Thurston-algorithm} can be found in \cite[Theorem~13.9]{TG}.
The idea of the proof is the following:
\begin{enumerate}
\item Construct a topological graph $G$ and a Markov action $L:G\to G$   such that  the incidence matrix $D_{(G,L)}$ is exactly $A_{\theta}$ (the matrix in the algorithm).
\item Construct a continuous, finite-to-one, and surjective semi-conjugacy $\Phi$ from $L:G\to G$ to $ f_{}:H_f\to H_f$, and then apply Proposition \ref{Do4}.
\end{enumerate}
 We may thus conclude that
\[\log\rho(A_{\theta})=\log\rho(D_{(G,L)})\overset{\text{Prop. \ref{entropy-formula}}}=h(G,L)\overset{\text{Prop. \ref{Do4}}}=h(H_f,f).\]

 In the higher degree case, the basic idea is similar, but there are several extra difficulties to overcome. The main problem is that  a  semi-conjugacy from $L:G\to G$ to $f:H_f\to H_f$ is much more difficult to construct, as we will see in the next section. The specific reason will be explained in the rest of this section.

 Let $f$ be a \pf polynomial with a weak critical marking $\Th=\{\Theta_1,\ldots,\Theta_{m+n}\}$.  Recall the definition of ${\rm post}(\Th)$ given in the table of Section \ref{section-critical-portrait}.

Let $G$  be a topological graph with
vertex set $V_G:=\text{post}(\Th),$ and  edge set
 $E_{G}:=\big\{e(x,y)|\ x\not=y\in V_{G}\big\}$. In other words, $G$ is an (undirected) topological  \emph{complete graph} with vertex set $V_G$. Note that when ${\rm post}(\Th)=\{x\}$ the graph $G$ consists of only one vertex $x$. In this case we specially consider $x$ as the (trivial) edge of $G$ to achieve consistency in the following statement.

Mimicking the action of the linear map $\AAA$ in the algorithm, we define a Markov map $L:G\rightarrow G$ as follows. For simplicity, for any $x\in V_G$, we set $e(x,x):=x$, a vertex of $G$.
Let $e(x,y)$ be an edge of $G$. If $x,y$  are non-separated by $\Th$,  let $L$ map $e(x,y)$ monotonously onto the edge or vertex $e(\tau(x),\tau(y))$ of $G$, with $x$ mapped to $\tau(x)$ and $y$ mapped to $\tau(y)$.
If $x,y$ has the separation set $(k_1,\ldots,k_p)\not=\emptyset$,
then subdivide the edge $e(x,y)$ into $p+1$ non-trivial arcs $\delta(z_i,z_{i+1}), i\in[0,p],$ with $z_0:=x$ and $z_{p+1}:=y$, and
let $L$ map each arc $\delta(z_i,z_{i+1})$ monotonously onto $e(\tau(\theta_i),\tau(\theta_{i+1}))$  with $z_i$ mapped to $\tau(\theta_i)$ and $z_{i+1}$ mapped to $\tau(\theta_{i+1})$, where $\theta_0:=x$, $\theta_{p+1}:=y$ and $\theta_i\in \Theta_{k_i}$ for each $i\in[1,p]$.

 It is easily seen that the edges of $G$ are in a one-to-one correspondence to the pairs in $\SSS$ under the correspondence $e(x,y)\to \{x,y\}$.
 From the definition of $L$,
 the incidence matrix of $(G,L)$ is exactly the matrix $A_\Th$ in the algorithm. It follows from Propositions \ref{Do3} and \ref{entropy-formula} that
 \begin{proposition}\label{equal-0}
The topological entropy $h(G,L)$ equals to $\log\rho(\Th)$.
\end{proposition}

So we only need to prove  $h(G,L)=h(H_f,f)$.
 Motivated  by Proposition \ref{Do4}, we try to construct a continuous, finite-to-one, and surjective semi-conjugacy  from $L:G\to G$ to $ f:H_f\to H_f$. A {\it priori}, the following way seems feasible. Divide the vertices of  $G$ into two parts:
\begin{equation}\label{V-G-F}
V_{G,\FFF}:= \text{post}(\Th_\FFF)\quad\text{and}\quad
V_{G,\JJJ}:= \text{post}(\Th_\JJJ),
\end{equation}
 where the definitions of $\text{post}(\Th_\FFF)$ and $\text{post}(\Th_\JJJ)$ are given in the table of Section \ref{section-critical-portrait}.  We remark that both of $V_{G,\FFF}$ and $V_{G,\JJJ}$ are finite since
the map $f$ is postcritically finite.
For each  vertex $x\in V_{G,\FFF}$, the ray $\RRR(x)$ supports a Fatou component, denoted by $U_x$, whereas for each  vertex $y\in V_{G,\JJJ}$, the ray $\RRR(y)$ lands at a postcritical point.   We are then prompted to first define a map
 $\chi:V_{G}\to {\rm post}(f)$ such that $\chi(x)$ is the center of $U_x$ if $x\in V_{G,\FFF}$ and $\chi(x)=\gamma(x)$ if $x\in V_{G,\JJJ}$, and then continuously extend $\chi$ to a map from $G$ to $H_f$, also denoted by $\chi$,  such that $\chi$ maps each edge $e(x,y)$ monotonously onto the regulated arc $[\chi(x),\chi(y)]\subset H_f$.

 This construction of $\chi$ indeed works in the quadratic case, so the idea of the proof shown above can be directly realized.  However, in the higher degree case, it may happen that the ray $\RRR(x)$ with $x\in V_{G,\FFF}$ supports simultaneously two Fatou components, or that an angle $y$ belongs to  $V_{G,\FFF}\cap V_{G,\JJJ}$ (see Example \ref{no-well}). It means that such $\chi$ is not always well defined, so that the construction of the projection from $G$ to $H_f$ described above need not work. This is the key point that causes difficulties of proving Theorem \ref{Thurston-algorithm} in the higher degree case.

 \begin{example}\label{no-well}
 The first example comes from \cite[Section I.4.4]{Poi1}. Consider the cubic polynomial $f(z)=z^3-\frac{3}{2}z$. It is easy to see that its critical points are $\pm\frac{\sqrt{2}}{2}$ and they are interchanged by $f$ (see the left one of Figure \ref{intersect}). We may choose a critical portrait of $f$ as
 \[\Th=\Bigl\{\ \Theta(U_A)=\{7/12,1/4\},\ \Theta(U_B)=\{11/12,1/4\}\ \Bigr\}.\]
 The ray $\RRR(1/4)$ supports the Fatou components $U_A$ and $U_B$ simultaneously.

 We then consider the cubic polynomial $f(z)=z^3+\frac{3}{2}z^2$.  One of its critical points $0$ is fixed, and another critical point $-1$ is mapped to a repelling fixed point $1/2$ (the right one of Figure \ref{intersect}). We may choose a critical portrait of $f$ as
 \[\Th=\Bigl\{\ \Theta(U)=\{0,1/3\};\ \Theta(-1)=\{1/3,2/3\}\ \Bigr\}.\]
 In this case, the angle $1/3$ belongs to both $\Theta(U)$ and $\Theta(-1)$.
 \begin{figure}
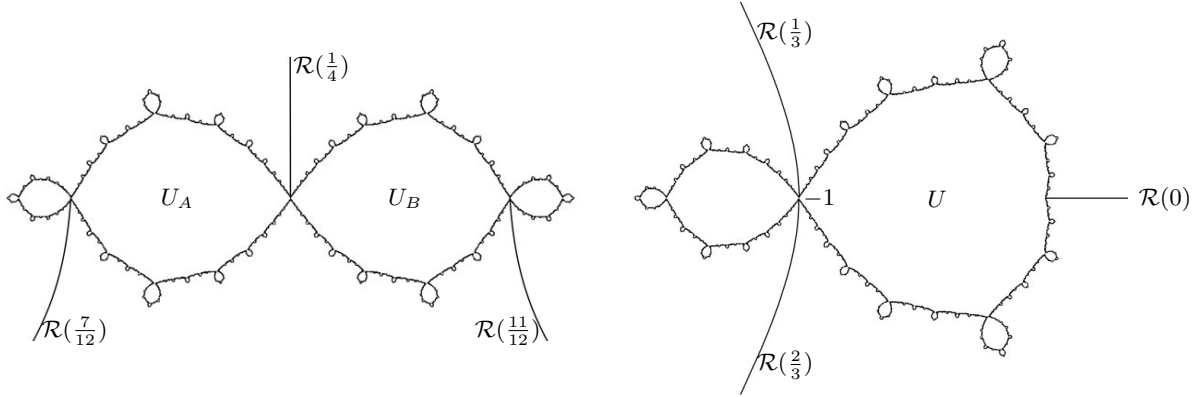

\begin{tikzpicture}
\node at (-4,0) {\includegraphics[width=8.5cm]{left.pdf}};
\node at (3.5,0) {\includegraphics[width=7cm]{right.pdf}};
\node at (-3.6,1.75){\footnotesize{$\RRR(\frac{1}{4})$}};
\node at (-6.85,-1.75){\footnotesize{$\RRR(\frac{7}{12})$}};
\node at (-1.15,-1.75){\footnotesize{$\RRR(\frac{11}{12})$}};
\node at (-5.5,0){\footnotesize{$U_A$}};
\node at (-2.5,0){\footnotesize{$U_B$}};
\node at (7.5,0){\footnotesize{$\RRR(0)$}};
\node at (2.5,2.2){\footnotesize{$\RRR(\frac{1}{3})$}};
\node at (2.5,-2.2){\footnotesize{$\RRR(\frac{2}{3})$}};
\node at (2.95,0){\footnotesize{$-1$}};
\node at (4.5,0){\footnotesize{$U$}};
\end{tikzpicture}
 \caption{Badly mixed cases of critical portraits}\label{intersect}
\end{figure}
 \end{example}

\section{The proof of Theorem \ref{Thurston-algorithm}}\label{proof}
 By Proposition \ref{equal-0}, we just need to verify the equality $h(G,L)=h(H_f,f)$.
The idea is as follows (using the notations in Section \ref{section-algorithm}).

To solve the problem that $\chi:G\to H_f$ is not well defined, we reduce the Hubbard tree $H_f$ to another finite tree $T$ by collapsing critical and postcritical internal rays so that $\phi\circ\chi:G\to T$ is well defined, where $\phi$ denotes the quotient map from $H_f$ to $T$. Meanwhile, the Markov map $f: H_f\to H_f$ descends to a Markov map $g:T\to T$ fulfilling that $\phi\circ f=g\circ \phi$. Then, to establish $h(H_f,f)=h(G,L)$,  on one hand we prove that $h(H_f,f)=h(T,g)$, which is based on an analysis of the incidence matrices of $(H_f,f)$ and $(T,g)$; on the other hand we verify that $h(T,g)=h(G,L)$, for which we basically use Proposition \ref{Do4}.
Following this strategy,  our proof is divided into several steps.

Throughout this section, we fix a \pf polynomial $f$ and a weak critical marking of $f$
\[\Th=\Bigl\{\Theta_1(c_1),\ldots,\Theta_m(c_m);\Theta(U_1),\ldots,\Theta(U_n)\Bigr\}\triangleq\{\Theta_1,\ldots,\Theta_{m+n}\}.\]


\subsection {The construction of the quotient map $\phi$}\label{1}

 Recall the definition of critical/postcritical internal rays given in Section \ref{section-critical-portrait}.
By collapsing the critical and postcritical internal rays, we reduce the Hubbard tree $H_f$ to a tree $T$.

Precisely,  the relation $\sim$ is defined on $\C$ such that $z\sim w$ if and only if $z=w$ or $z$ and $w$ are contained in  a path constituted by critical or postcritical internal rays.
It is clear that $\sim$ is an equivalence relation on $\C$. We denote by $\phi:\C\to\C/_\sim$ the quotient map.
\begin{lemma}\label{quotient}
The quotient space $\C/_\sim$ is homeomorphic to $\C$.
\end{lemma}
\begin{proof}
Note that each $\sim$-class with more than one point is a finite tree, so the relation $\sim$ clearly satisfies the following $4$ properties:
\begin{enumerate}
\item there are at least two
  distinct equivalence classes;
\item it is closed as a subset of $\widehat{\C}\times \widehat{\C}$ equipped with
the product topology;
\item each equivalence class is a compact connected set;
\item the complementary component of each equivalence class is connected.
\end{enumerate}
Moore's Theorem (\cite{M}) says that if an equivalence relation on $\widehat{\C}$ has these $4$ properties, then the quotient space of $\widehat{\C}$ modulo this equivalence relation is homeomorphic to $\widehat{\C}$. By applying the theorem to the one-point
compactification of $\C$, we get that $\C/_\sim$ is homeomorphic to $\C$.
\end{proof}

For simplicity, we assume that $\C/_\sim=\C$, and $\phi$ is chosen to be  identity outside a small neighborhood of the filled in Julia set $K_f$.

\begin{definition}[essential $\sim$-equivalence classes]\label{essential}
 A $\sim$-equivalence class $\x$ is called \emph{essential} if it is either \emph{non-trivial}, meaning that it contains more than one point, or a point in $ {\rm crit}(f)\cup  {\rm post}(f)$.
\end{definition}

Thus $\phi(z)=\phi(w)$ with $z\not=w\in\C$ if and only if $z,w$ belong to an essential $\sim$-class. Note also that each  essential $\sim$-equivalence class is either a finite tree or a point. In the latter case, it is a critical or postcritical point in the Julia set.




We now define $T:=\phi(H_f)$. Since each non-trivial $\sim$-equivalence class $\x$  is a regulated tree in $K_f$, then, for any regulated tree $H\subset K_f$, we have $[z,w]\subset \x\cap H$ if $z,w\in \x\cap H$. The intersection of each $\x$ and $H$ is hence either empty, or one point or a sub-tree of $H$. It implies the proposition below.
\begin{proposition}\label{finite-tree}
The restriction of $\phi$ on any regulated tree within $K_f$ is monotone. In particular, $\phi:H_f\to T$ is monotone and $T$ is a finite tree in $\C$.
\end{proposition}

The vertex set of $T$ is defined to be $V_T=\phi(V_{H_f})$. To characterize the edge set of $T$, we define a subset $E^{\text{col}}_{H_f}$ of the edge set of $H_f$ by
\[E_{H_f}^{\text{col}}=\{e\in E_{H_f}\mid e\text{ is contained in an essential $\sim$-class}\}.\]
Then the edge set of $T$ is equal to $\{\phi(e)\mid e\in E_{H_f}\setminus E_{H_f}^{\text{col}}\}$.

\begin{lemma}\label{star-like}
There is an universal constant $N>0$, such that for any essential $\sim$-class $\xi$ and all integer $n>N$, the set $f^n(\x)$ is either a periodic point in $J_f$, or a star-like tree consisting of periodic internal rays.
\end{lemma}
\begin{proof}
Let $\x$ be any essential $\sim$-class. If $\x$ is a point, it is a critical or postcritical point in the Julia set, and hence $f^n(\x)$ is a periodic point in $J_f$ for sufficiently large $n$.

We now assume that $\x$ is not a point. It is then the union of several critical or postcritical internal rays. Note that such internal rays are eventually periodic by iterations of $f$ (Lemma \ref{basic-property}.(3)), so for any sufficiently large $n$, the set $f^n(\x)$ is a tree consisting of periodic postcritical internal rays. By Lemma \ref{basic-property}.(5), distinct internal rays in $f^n(\x)$ lie in different Fatou components. It implies that $f^n(\x)$ is a star-like tree.
\end{proof}



 \subsection {The construction of a Markov map $g:T\to T$}\label{2}

 Clearly, the equivalence relation $\sim$ defined in Section \ref{1} is $f$-invariant, i.e., $x\sim y\Rightarrow f(x)\sim f(y)$. The polynomial $f:\C\to\C$ therefore descends to a map $f/_\sim:\C\to \C$ by $\phi$. Set $g:=f/_\sim$.  The following commutative graph holds.
  \begin{equation}\label{commutative0}\begin{array}{rcl}\C &\xrightarrow[]{\ f\ }  &\C \\ \phi\Big\downarrow &&\Big\downarrow \phi  \\ \C &  \xrightarrow[]{\ g\ } & \C\vspace{-0.1cm}.\end{array}
 \end{equation}
Note that $g\circ\phi=\phi\circ f$ is continuous, hence $g$ is continuous by the functorial property of the quotient topology.

 We call ${\rm crit}(g):=\phi( {\rm crit}(f))$ the \emph{critical set} of $g$, and ${\rm post}(g):=\phi( {\rm post}(f))$ the \emph{postcritical set} of $g$. For any $\theta\in \T$, as $\phi$ is injective within the unbounded Fatou component, it maps the external ray $\RRR_f(\theta)$ to a simple curve, denoted by $\RRR_g(\theta)$, called the \emph{external ray of $g$ with angle $\theta$}.
 Similarly, we denote by $\g_g(\theta)$ the landing point of $\RRR_g(\theta)$. Clearly, we have $\g_g(\theta)=\phi(\g_f(\theta))$ for every $\theta\in \T$.
With the semi-conjugacy $\phi$, the map $g$ inherits many properties of the map $f$.

\begin{proposition}\label{property-g}
\begin{enumerate}
\item  ${\rm post}(g)=\cup_{n\geq1}g^n({\rm crit}(g)), \ g(V_T)\subset V_T$ and   ${\rm post}(g)\subset V_T$.
\item For each $\theta\in \T$, we have $g(\RRR_g(\theta))=\RRR_g(\tau(\theta))$ and $g(\g_g(\theta))=\g_g(\tau(\theta))$.
\item  For each $\theta\in {\rm crit}(\Th)$, \emph{resp.} ${\rm post}(\Th)$, the point $\g_g(\theta)$  belongs to ${\rm crit}(g)$, \emph{resp}. ${\rm post}(g)$.
\item The tree $T$ is $g$-invariant and the  map $g:T\to T$ is Markov.
\end{enumerate}
\end{proposition}
\begin{proof}
{(1), (2).}  One just need to notice that ${\rm crit}(g),{\rm post}(g)$ and $V_T$ are the images of $ {\rm crit}(f), {\rm post}(f)$ and $V_{H_f}$ by $\phi$. To complete the argument is straightforward using the commutative graph (\ref{commutative0}).

{\noindent (3).}  For any $\theta\in {\rm crit}(\Th)$, \emph{resp.} ${\rm post}(\Th)$, if $\theta\in {\rm crit}(\Th_\FFF)$, \emph{resp.} ${\rm post}(\Th_\FFF)$, by Lemma \ref{basic-property}.(1), there exists a Fatou component $U$ such that a critical internal ray, \emph{resp.} postcritical internal ray $r_{U}(\theta)$ joins $\g_f(\theta)$ and a critical point \emph{resp.} postcritical point of $f$ (the center of $U$); and  if $\theta\in {\rm crit}(\Th_\JJJ)$, \emph{resp.} ${\rm post}(\Th_\JJJ)$, then $\g_f(\theta)\in{\rm crit}(f)$, \emph{resp.} ${\rm post}(f)$. In both case $\phi$ maps $\g_f(\theta)$ to a point of ${\rm crit}(g)$, \emph{resp.} ${\rm post}(g)$.

{\noindent (4).} From the commutative graph (\ref{commutative0}), we get $g(T)=\phi\circ f(H_f)\subset\phi(H_f)=T$. By Proposition \ref{Hubbard-tree}, the tree $H_f$ can be broken into a system of arcs $\Delta_{H_f}$ such that the restriction of $f$ on each one of $\Delta_{H_f}$ is homeomorphic to an edge of $H_f$. Projected by $\phi$, the set of arcs
$$\Delta_T:=\{\phi(\delta)\mid \delta\in\Delta_{H_f}\text{ but } \delta\nsubseteqq\text{ any essential $\sim$-class}\}$$
forms a decomposition of $T$. And by the commutative graph (\ref{commutative0}), the map $g$ maps each one of $\Delta_T$  either to a vertex of $T$ or monotonously onto an edge of $T$. Then $g:T\to T$ is  Markov.
\end{proof}

\subsection{The entropies $h(H_f,f)$ and $h(T,g)$ are equal}

 We have shown that $f:H_f\to H_f$ and $g:T\to T$ are both Markov maps, so, by Proposition \ref{entropy-formula}, it is enough to show that the spectral radii of the
 incidence matrices $D_{(H_f,f)}$ and $D_{(T,g)}$ are equal. To do this, we need the lemma below.

Recall that $E_{H_f}^{\text{col}}=\{e\in H_f\mid e\text{ is contained in an essential $\sim$-class}\}$. By a \emph{cycle} in $E_{H_f}^{\text{col}}$ we mean a subset of $E_{H_f}^{\text{col}}$
in the form \vspace{-5pt}
\[O=\{e_0=e_n,e_1=f(e_0),\ldots, e_n=f(e_{n-1})\}\]

\begin{lemma}\label{f-invariant}
All edges in $E_{H_f}^{\text{col}}$ are attracted by cycles, i.e., for each $e\in E_{H_f}^{\text{col}}$, the sequence of
 iterates $f^n(e)$ ($n\geq0$) eventually falls on the union of cycles in $E_{H_f}^{\text{col}}$.
\end{lemma}
\begin{proof}
Let $e\in E_{H_f}^{\text{col}}$. It is by definition the union of several critical or postcritical internal rays. So is $f(e)$. It follows that $f(e)$ is the union of edges in $E_{H_f}^{\text{col}}$.
Notice that each internal ray in $e$ contains a critical or postcritical point, which belongs to $V_{H_f}$,  then the edge $e$, and hence $f(e)$, is  constituted by one or two critical/postcritical internal rays. It follows that $f(e)$ is either still an edge in $E_{H_f}^{\text{col}}$ or the union of two edges in $E_{H_f}^{\text{col}}$ which are both postcritical internal rays.
Repeating this process, we encounter the following two cases:
 \begin{enumerate}
 \item All $f^n(e)$ are edges. In this case $e$ is eventually periodic because $H_f$ is a finite tree.
 \item For some $k\geq1$, the set $f^k(e)$ is the union of two edges in $E_{H_f}^{\text{col}}$, which are both postcritical internal rays. Since every postcritical internal ray is eventually periodic, the iteration of $e$ finally falls on the union of cycles  in $E_{H_f}^{\text{col}}$.
 \end{enumerate}
Then the lemma is proven.
\end{proof}

\begin{proposition}\label{equal}
The entropies $h(H_f,f)$ and $h(T,g)$ are equal.
\end{proposition}
\begin{proof}
Set $$Y:=\bigcup\ \{e\mid e\in E_{H_f}^{\text{col}}\}.$$
Then $f(Y)\subset Y$ by Lemma \ref{f-invariant}, and  $f:Y\to Y$ is a Markov map. Denoting the incidence matrix of $(Y,f)$ by $M$, it follows from Proposition \ref{entropy-formula} that $h(Y,f)=\log\rho(M)$.

Arranging the edges of $H_f$ in the order  $E_{H_f}^{\text{col}},\ E_{H_f}\setminus E_{H_f}^{\text{col}}$,
 the incidence matrix of $(H_f,f)$ takes the form
\[D_{(H_f,f)}=\left(\begin{array}{cc}
M&\star\\
\mathbf{0}&B
\end{array}\right)
\]
where $\mathbf{0}$ denotes a zero-matrix.
Note that the matrix $B$  is exactly the incidence matrix of $(T,g)$, so it is enough to prove  $\rho(M)=1$, or, equivalently, $h(Y,f)=0$.

 We denote by $O_1,\ldots, O_k$ the cycles in $E_{H_f}^{\text{col}}$, and set $O=\bigcup_{i=1}^k O_i$. Clearly $f(O)\subset O$ and $f:O\to O$ is a Markov map.  Using Lemma \ref{f-invariant}, Propositions \ref{Do3} and \ref{entropy-formula}, we have
 \[h(Y,f)=h(O,f)=\log\rho(D_{(O,f)}).\]
It suffices  to show that $\rho(D_{(O,f)})=1$.

Group the edges of $H_f$ in $O$ in the order $O_1,\ldots, O_k$. Then the incidence matrix $D_{(O,f)}$ takes the form
\[\left(\begin{array}{ccc}
M_1&&\\
&\ddots&\\
&&M_k
\end{array}\right)
\]
Fix $j\in[1,k]$. Denote the cycle of edges in $O_j$ by
$e_0\mapsto e_1\mapsto\cdots\mapsto e_n=e_0.$
If we further arrange the edges in $O_j$ in the order
$e_0,\ e_1,\ \ldots,\ e_n,$
 the matrix $M_{j}$ takes the form
\[\left(\begin{array}{ccccc}
0&&&1\\
1&0&&&\\
&\ddots&\ddots&\\
&&1&0\\
\end{array}\right)
\]
Clearly, $\rho(M_{j})=1$.
\end{proof}


\subsection{The partition induced by $\Th$ in the $g$-plane}\label{partition-in-g}

In Section \ref{partition-in-f} we gave the partition  in the $f$-plane induced by $\Th$.
 Using the projection $\phi$, we can transfer the partition in the $f$-plane to the $g$-plane.

 For any $\Theta\in\Th$, since $\phi$ collapses each critical and postcritical internal ray, then, by the construction of $\RRR_f(\Theta)$ given in \eqref{eq:1}, all external rays in the $g$-plane with arguments in $\Theta$ land at a common critical point
\begin{equation}\label{eq:critical}
c_\Theta:=\phi(\RRR_f(\Theta)\cap K_f)
\end{equation}
of $g$, and
\begin{equation}\label{intersection}
\RRR_g(\Theta):=\phi(\RRR_f(\Theta))=\big(\cup_{\theta\in\Theta}\RRR_g(\theta)\big)\cup\{c_\Theta\}.
\end{equation}
Similarly, we can also define the unlinked equivalence relation in the $g$-plane.

\begin{definition}[unlink equivalent in the $g$-plane]
 Two points $z,w\in \C\setminus\cup_{i=1}^{m+n}\RRR_g(\Theta_i)$ are said to be \emph{unlink equivalent in the $g$-plane} if they belong to a common component of $\C\setminus \RRR_g(\Theta_i)$ for every $\Theta_i\in\Th$.
\end{definition}

Note that the intersection of each $\RRR_f(\Theta_i)$ and $K_f$ is either a critical point or the union of critical internal rays, then the following result is straightforward.

\begin{proposition}\label{partition-g}
Let $V_1,\ldots,V_d$ be the unlinked equivalence classes in the $f$-plane.
Set $W_i:=\phi(V_i), i=1,\ldots,d$. Then each $W_i$ is an unlinked equivalence class in the $g$-plane. And we have $\ov{W_i}=\phi(\ov{V_i})$ and $\partial W_i=\phi(\partial V_i)$.
\end{proposition}

With these preparations, we can prove the key lemma below. Note that for any angle $x\in V_G$, by (1) and (3) of Proposition \ref{property-g}, the landing point $\g_g(x)$ of the ray $\RRR_g(x)$ belongs to $T$. 

\begin{lemma}\label{injective-g}
 Let $x,y$ be two distinct points in $V_G$.
  \begin{enumerate}
  \item If $x,y$ are separated by $\Theta\in \Th$, then $[\g_g(x),\g_g(y)]_T$ intersects the critical point $c_\Theta$ of $g$ given in \eqref{eq:critical}.
  \item If $x,y$ are non-separated by $\Th$, then $\g_g(x)$ and $\g_g(y)$ belong to the closure of  an unlinked equivalence class  in the $g$-plane, and $g$ maps $[\g_g(x),\g_g(y)]_T$ monotonously onto $[\g_g(\tau(x)),\g_g(\tau(y))]_T$.
  \end{enumerate}
\end{lemma}
\begin{proof}
\begin{enumerate}
\item By Proposition \ref{finite-tree}, the map $\phi$ is monotonous when restricted on $[\g_f(x),\g_f(y)]$, so $\phi\big([\g_f(x),\g_f(y)]\big)=[\g_g(x),\g_g(y)]_T$ by Proposition \ref{tree-map}.  Since $c_\Theta=\phi(\RRR_f(\Theta)\cap K_f)$, we remain to show that
$\RRR_f(\Theta)\cap[\g_f(x),\g_f(y)]\not=\emptyset$.

By looking at the circle at infinity,  the rays $\RRR_f(x)$ and $\RRR_f(y)$ lie in different components of $\C\setminus\RRR_f(\Theta)$.  As $\RRR_f(\Theta)$ is connected, then the intersection of $\RRR_f(\Theta)$ and $[\g_f(x),\g_f(y)]$ is non-empty.

 \item By Lemma \ref{separation}, there exist an unlink equivalence $I_k$ on $\T$ and two elements $\Theta_i,\Theta_j\in \Th$ intersecting $\ov{I_k}$ such that $x,y\in I_k\cap\Theta_i\cap\Theta_j$. Correspondingly, in the dynamical plane of $f$, the points $\g_f(x)$ and $\g_f(y)$ belong to $\ov{V_k}\cup\RRR_f(\Theta_i)\cup\RRR_f(\Theta_j)$. Note that $\phi(\RRR_f(\Theta_i)\cap K_f)$ and $\phi(\RRR_f(\Theta_j)\cap K_f)$ are points contained in $\phi(\ov{V_k})$, then $\g_g(x)$ and $\g_g(y)$ belong to $\ov{W_k}=\phi(\ov{V_k})$.

To prove the latter part of (2), note that $g(\g_g(\theta))=\g_g(\tau(\theta))$ for any $\theta\in \T$ (Proposition \ref{property-g}), so, using Proposition \ref{tree-map}, we just need to prove that $g|_{[\g_g(x),\g_g(y)]}$ is monotone.

Let $W$ be an unlinked equivalence class in the $g$-plane with $\g_g(x),\g_g(y)\in\ov{W}$. Since $\g_g(x),\g_g(y)$ also belong to $ T$,  there exist two points $z,w\in \ov{V}\cap H_f$ such that $\phi(z)=\g_g(x)$ and $\phi(w)=\g_g(y)$, where $V$ is the unlinked equivalence class in the $f$-plane with $\phi(V)=W$. As $\phi|_{[z,w]}$ is monotone (Proposition \ref{finite-tree}), then, by Proposition \ref{tree-map}, we have $\phi([z,w])=[\g_g(x),\g_g(y)]_T$.

We apply the formula $\phi\circ f=g\circ \phi$ on $[z,w]$, and use the notations in the proof of  Lemma  \ref{injective-f}.(4). Recall that $A_1,\ldots,A_m$ denotes
the closures of the connected component of $[z,w]\setminus\partial  V$ such that each $A_i$ separates $A_{i-1}$ and $A_{i+1}$, and $E_i=f(A_i)$ for $i=1,\ldots,m$.
Note first that $[\g_g(x),\g_g(y)]_T=\phi([z,w])=\phi(\bigcup_{i=1}^m A_i)$ and
\begin{equation}\label{eq*}
g([\g_g(x),\g_g(y)]_T)=g\circ\phi(\cup_{i=1}^m A_i)=\phi\circ f(\cup_{i=1}^m A_i)=\phi(\cup_{i=1}^m E_i).
\end{equation}
Let $y$ be any point in $g([\g_g(x),\g_g(y)]_T)$. We will show that $(g|_{[\g_g(x),\g_g(y)]})^{-1}(y)$ is connected. Note that the fiber $Y:=\phi^{-1}(y)$ is either a point or a regulated tree in $K_f$. Denote by $L$ the regulated arc containing all $E_i, i=1,\ldots,m,$ constructed in the proof of Lemma \ref{injective-f}.(4). Then
\begin{equation}\label{eq4}
Y\cap \big(\cup_{i=1}^m E_i\big)=(Y\cap L)\cap \big(\cup_{i=1}^m E_i\big)=R\cap \big(\cup_{i=1}^m E_i\big),
\end{equation}
where $R:=Y\cap L$  is  a segment in $L$. There is thus a segment $R'$ in $[z,w]$ such that $f(R'\cap \big(\cup_{i=1}^m A_i\big))=R\cap \big(\cup_{i=1}^m E_i\big)$. Since $f$ is injective on $\bigcup_{i=1}^m A_i$, then
\begin{equation}\label{eq5}
f^{-1}(R\cap \big(\cup_{i=1}^m E_i\big))\cap \big(\cup_{i=1}^m A_i\big)=R'\cap \big(\cup_{i=1}^m A_i\big).
\end{equation}
By (\ref{eq*}), (\ref{eq4}) and (\ref{eq5}), the set  $\phi(R'\cap \bigcup_{i=1}^m A_i)$ is exactly the fiber $$(g|_{[\g_g(x),\g_g(y)]_T})^{-1}(y)=g^{-1}(y)\cap [\g_g(x),\g_g(y)]_T.$$
And  it is by the construction of $A_i$ either a point or a closed segment. Thus, $g|_{[\g_g(x),\g_g(y)]_T}$ is monotone
as we wishes to show.
\end{enumerate}
\end{proof}

\subsection{The construction of a projection  $\Phi:G\to T$}\label{6}

Recall that $G$ is a topological complete  graph with the
vertex set ${\rm post}(\Th)$.
Proposition \ref{property-g}.(3) says that each external ray in the $g$-plane with argument in $V_G$ lands at a postcritical point of $g$.  We then define a map $\Phi:V_G\to {\rm post}(g)$ by $\Phi(x)=\g_g(x)$ for $x\in V_G$, where $\g_g(x)$ denotes the landing point of $\RRR_g(x)$.

\begin{lemma}\label{property-phi}
The map $\Phi$ is surjective, and satisfies $g\circ\Phi=\Phi\circ\tau$.
\end{lemma}
 \begin{proof}
By the construction of $\Th$, for each postcritical point $z$,  there exists an angle $x\in V_G$ such that either $\g_f(x)=z$ or a postcritical internal ray $r_{U}(x)$ joins $z$ and $\g_f(x)$.  In both cases, we have
$\phi(z)=\phi(\g_f(x))=\g_g(x)=\Phi(x).$
Since ${\rm post}(g)=\phi({\rm post}(f))$, then $\Phi$ is surjective.
 By Proposition \ref{property-g}.(2), we get
$g\circ\Phi(x)=g(\g_g(x))=\g_g(\tau(x))=\Phi\circ\tau(x)$ for each $x\in V_G$.
\end{proof}

 We continuously extend $\Phi$ to a map, also denoted by $\Phi$,  from $G$ to $T$ such that $\Phi$ maps each edge $e(x,y)$ of $G$ monotonously onto $[\Phi(x),\Phi(y)]_T$.
The case that $\Phi(x)=\Phi(y)$ may happen, in which  $[\Phi(x),\Phi(y)]_T$  reduces to a point. We thus obtain a projection from $G$ to $T$.

 \begin{proposition}\label{surjective}
 The projection $\Phi:G\to T$ is surjective.
 \end{proposition}
\begin{proof}
The Hubbard tree is by definition the regulated hull of postcritical points. It follows that the endpoints of $H_f$ belong to $ {\rm post}(f)$. Note that $\phi$ maps the endpoints of $H_f$ onto those of $T$, then the endpoints of $T$ belong to ${\rm post}(g)$. Since $G$ is a complete graph,
the arc $[p,q]_T$ is contained in $\Phi(G)$ for any endpoints $p,q$ of $T$.
We only need to invoke the fact that each edge of $T$ is contained in a regulated arc $[p,q]$ with $p,q$ two endpoints.
\end{proof}

In the construction, the projection $\Phi$ is only required to be monotone on each edge of $G$, so it is not necessarily a semi-conjugacy from $L:G\to G$ to $g:T\to T$. One may ask whether we can  impose some additional  conditions on the extension  such that $\Phi$ becomes a semi-conjugacy.
 The answer might be yes, but we will not work on this. One reason is that to change $\Phi$ into a semi-conjugacy, we need to carefully  analyse where $g$ is not injective on each edge of $T$ and  correspondingly modifying $L$ piecewise on each edge of $G$, which is a tedious work. The other reason, which is crucial, is that even if we modify $\Phi$ to a semi-conjugacy, it is  generally not finite to one, because we can only require that $\Phi$ is monotone, but not a homeomorphism, restricted on each edge of $G$, so that Proposition \ref{Do3} is not available.

 Instead, we will suitably modify $\Phi$ on each edge of $G$ such that the relation $g\circ \Phi=\Phi\circ f$ is satisfied in a weaker sense (see Lemma \ref{commutative2} below).


Recall the definition $L$  in Section \ref{T-A}.
If we specify that a separation set with the form $(k_1,\ldots,k_0)$ is empty, then the action of $L$ on an edge $e(x,y)\in E_G$ can be uniformly expressed as follows:

Let the ordered pair $x,y$ have the separation set $(k_1,\ldots,k_p), p\geq0$, i.e., the leaf $\ov{xy}$ successively crosses $hull(\Theta_{k_1}),\ldots,hull(\Theta_{k_p})$ from $x$ to $y$, and no other elements of $\Th$ separate $x$ and $y$. Subdivide the edge $e(x,y)$ into $p+1$  arcs $\delta(z_i,z_{i+1}), i\in[0,p],$ with $z_0:=x$ and $z_{p+1}:=y$, such that these arcs are non-trivial if $x\not=y$. And then
let $L$ map each arc $\delta(z_i,z_{i+1})$ monotonously onto $e(\tau(\theta_i),\tau(\theta_{i+1}))$, where $\theta_0:=x$, $\theta_{p+1}:=y$ and $\theta_i\in \Theta_{k_i}$ for each $i\in[1,p]$.

\begin{definition}[subdivision arcs of $G$]\label{subdivision}
For any edge $e(x,y)$ of $G$ with $x\not=y$, a non-trivial arc $\delta(z_i,z_{i+1})\subset e(x,y)$ described above is called a subdivision arc of $G$.
\end{definition}

 For a subdivision arc of $G$, we see from the action of $L$ that it is mapped either  monotonously onto an edge of $G$ or onto a vertex of $G$.


\begin{lemma}\label{commutative2}
We can modify the projection $\Phi$ on each subdivision arc $\delta$ of $G$ such that $\Phi|_\delta$ is either injective or a constant map, and the following equation holds:
$$g\circ \Phi(\delta)=\Phi \circ L(\delta).$$
\end{lemma}
\begin{proof}
Let $e(x,y)$ be any edge of $G$, and the ordered pair $x,y$ have the separation set $(k_1,\ldots,k_p)$ with $p\geq0$.  Then $e(x,y)$ contains $p+1$ subdivision arcs $\delta(z_i,z_{i+1}), i=0,\ldots p,$ with $z_0:=x$ and $z_{p+1}:=y$. We set $\theta_0:=x,\theta_{p+1}:=y$, and pick an angle $\theta_i$ in each $\Theta_{k_i}\in \Th$.

By Lemma \ref{injective-g}.(1), the arc $[\Phi(x),\Phi(y)]_T=[\g_g(x),\g_g(y)]_T$
(possibly reduced to a point) successively passes through the points
\[c_{\Theta_{k_i}}=\phi(\RRR_f(\Theta_{k_i})\cap K_f)=\g_g(\theta_i)=\Phi(\theta_i),\quad i=1,\ldots,p.\]
  It follows that $[\Phi(x),\Phi(y)]_T$ also contains $p+1$ successive subdivision sets $[\Phi(\theta_i),\Phi(\theta_{i+1})]_T, i=0,\ldots p,$ where each of them is either an arc or a point.

  We now modify $\Phi:e(x,y)\to [\Phi(x),\Phi(y)] $ on each subdivision arc of $G$ in $e(x,y)$ such that \begin{enumerate}
\item $\Phi(\delta(z_i,z_{i+1}))=[\Phi(\theta_i),\Phi(\theta_{i+1})]$ with $\Phi(z_i)=\Phi(\theta_i)$ and $\Phi(z_{i+1})=\Phi(\theta_{i+1})$, for each $i=0,\ldots,p$.
\item If $\Phi(\theta_i)\not=\Phi(\theta_{i+1})$, we require that $\Phi:\delta(z_i,z_{i+1})\to [\Phi(\theta_i),\Phi(\theta_{i+1})]$ is a homeomorphism.
\end{enumerate}
   By Corollary \ref{non-separated} and Lemma \ref{injective-g}.(2), we get that $$g[\Phi(\theta_i),\Phi(\theta_{i+1})]_T=g[\g_g(\theta_i),\g_g(\theta_{i+1})]_T=[\g_g(\tau(\theta_i)),\g_g(\tau(\theta_{i+1}))]_T=
[\Phi(\tau(\theta_i)),\Phi(\tau(\theta_{i+1}))]_T$$
Therefore, after this modification, we have
\begin{eqnarray*}
g\circ\Phi(\delta(z_i,z_{i+1}))&=&g[\Phi(\theta_i),\Phi(\theta_{i+1})]_T=[\Phi(\tau(\theta_i)),\Phi(\tau(\theta_{i+1}))]_T\\
&=&\Phi\big(e(\tau(\theta_i),\tau(\theta_{i+1}))\big)=\Phi\circ L(\delta(z_i,z_{i+1})),
\end{eqnarray*}
which completes the proof.
\end{proof}

\subsection{The construction of a projection $\Psi:\G\to T$}\label{7}

{ To solve the problem that $\Phi:G\to T$ is generally not finite to one, we will in this part construct a quotient graph $\G$ of $G$ and a finite to one map $\Psi:\G\to T$.}
\begin{definition}[collapsing subdivision arc]
A subdivision arc of $G$ is called a \emph{collapsing subdivision arc} if it is mapped by $\Phi$  to a point.
\end{definition}
Intuitively, by collapsing each collapsing subdivision arc to one point, we obtain a quotient graph $\G$, and the projection $\Phi:G\to T$ descends to a projection $\Psi:\G\to T$, which is injective on each edge of $\G$. We explain these facts in the following.


We define a relation $\simeq$ on $G$ such that $p\simeq q$ if and only if either $p=q$ or $p$ and $q$ are contained in a path constituted by collapsing subdivision arcs of $G$. This relation is obviously an equivalence relation, so we denote by $\Gamma:=G/_\simeq$ the quotient space and by $\wp:G\to \G$ the quotient map.

\begin{proposition}\label{quotient-graph}
The topological space $\G$ is a topological graph, and the Markov map $L:G\to G$ descends to a Markov map $Q:\G\to \G$ by $\wp$.
 \begin{equation}\label{commutative4}\begin{array}{rcl}G &\xrightarrow[]{\ L\ }  &G \\ \wp\Big\downarrow &&\Big\downarrow \wp  \\ \G &  \xrightarrow[]{\ Q\ } & \G\vspace{-0.1cm}.\end{array}
 \end{equation}
\end{proposition}
\begin{proof}
We define $V_{\G}:=\wp(V_G)$ the vertex set of $\G$ and
$E_{\G}:=\{\wp(\wt{e})\mid \wt{e}\in E_G\setminus E_G^{\text{col}}\}$
the edge set of $\G$, where $E_G^{\text{col}}$ denotes the set of the edges of $G$ which are constituted by collapsing subdivision arcs. It is not difficult to check that the topological space $\G$ with vertex set $V_{\G}$ and edge set $E_{\G}$ satisfies the properties of being a topological graph (see Section \ref{entropy-result}).

Let $\wt{\delta}$ be any collapsing subdivision arc of $G$. By Lemma \ref{commutative2}, we have that  $\Phi\big(L(\wt{\delta})\big)=g\big(\Phi(\wt{\delta})\big)$ is a singleton. Then $L(\wt{\delta})$ is either a point or the union of some collapsing subdivision arcs. It means that the equivalence relation $\simeq$ is $L$-invariant. The map $L:G\to G$ hence descends to a continuous self map on $\G$, which is denoted by $Q$.

Clearly, the set of arcs
\begin{equation}\label{eq:subdivision-arc}
\{\text{$\wp(\wt{\delta})$: $\wt{\delta}$ is a subdivision arc, but not a collapsing subdivision arc of $G$}\}
\end{equation}
 forms a system of subdivision arcs of $\G$. By the formula $\wp\circ L=Q\circ \wp$ on $G$, the restriction of $Q$ on such a subdivision arc is either monotone onto an edge of $\G$, or a constant map.
 Hence $Q:\G\to \G$ is a Markov map.
\end{proof}

\begin{proposition}\label{semi-conjugacy}
 There exists a surjective and finite to one map $\Psi:\G\to T$ such that $\Psi\circ \wp=\Phi$ pointwise on $G$. 
\begin{equation}\label{commutative-3}
\xymatrix@R=0.5cm{
  G \ar[dd]_{\Phi} \ar[dr]^{\wp}             \\
                & \G \ar[dl]_{\Psi}         \\
  T               }
\end{equation}
\end{proposition}
\begin{proof}
 For any $p\in \G$, we define $\Psi(p)=\Phi\circ\wp^{-1}(p)$. Note that $\wp^{-1}(p)$ is either a point or a connected set consisting of collapsing subdivision arcs. Since $\Phi$ maps each collapsing subdivision arc to a point,  then $\Phi\big(\wp^{-1}(p)\big)$ is always a singleton. It means that $\Psi$ is well defined. The surjection of $\Psi$ follows directly from the surjectivity of $\Phi$ given in Proposition \ref{surjective}.


To prove that $\Psi$ is finite to one, it is enough to show that its restriction on each edge of $\G$ is injective. Let $e\in E_\G,\ p,q\in e$ and $\Psi(p)=\Psi(q)$. By the definition of $E_\G$, we can pick an edge $\widetilde{e}$ of $G$ in $E_G\setminus E_G^{\text{col}}$ such that $\wp(\wt{e})=e$. Denote by $\tilde{p}$ and $\tilde{q}$,
 preimages on $\wt{e}$ by $\wp$ of $p$ and $q$, respectively.
 We then have $\Phi(\tilde{p})=\Phi(\tilde{q})$. According to the modified construction of $\Phi$ (Proposition \ref{commutative2}), it follows that $\tilde{p}$ and $\tilde{q}$ belong to a path constituted by collapsing subdivision arcs of $G$. Hence $p=\wp(\tilde{p})=\wp(\tilde{q})=q$. It means that $\Psi|_e$ is injective.
\end{proof}

\subsection{The entropies $h(\G,Q)$ and $h(T,g)$ are equal}\label{8}

{By Proposition \ref{semi-conjugacy}, the map $\Psi:\G\to T$ is surjective and finite to one. To prove $h(\G,Q)=h(T,g)$, we only need to
redefine $Q$ on each subdivision arcs of $\G$ such that $\Psi$ is a semi-conjugacy from $Q:\G\to\G$ to $g:T\to T$, and then apply Proposition \ref{Do4}.}

\begin{proposition}\label{equal-1}
The topological entropies $h(\G,Q)$ and $h(T,g)$ are equal.
\end{proposition}
\begin{proof}
 In \eqref{eq:subdivision-arc} we obtain a system of subdivision arcs for $\G$. Let $\delta=\wp(\wt{\delta})$ be such one. By the definition of the relation $\simeq$, the map $\wp:\wt{\delta}\to\delta$ is a homeomorphism.
By Proposition \ref{commutative2}, and the commutative graphs (\ref{commutative4}) and (\ref{commutative-3}), we have
$$g\circ\Psi(\delta)=g\circ\Phi(\wt{\delta})=\Phi\circ L(\wt{\delta})=\Psi\circ Q(\delta)$$
If $Q(\delta)$ is a singleton, the formula $g\circ\Psi=\Psi\circ Q$ naturally holds pointwise on $\delta$. Otherwise, $Q(\delta)$ is an edge of $\G$, and the maps $\Psi|_{\delta}$ and $\Psi|_{Q(\delta)}$ are both homeomorphisms onto their images. We then redefine $Q$ on $\delta$  by lifting $g:\Psi(\delta)\to \Psi(Q(\delta))$ along these two homeomorphisms, i.e. by setting   $ Q:=\Psi^{-1}\circ g\circ\Psi$ on $\delta$.
After such  modification of $Q$ on each subdivision arc of $\G$, we obtain the formula $g\circ\Phi=\Phi\circ Q$ pointwise on $\G$. By Proposition \ref{entropy-formula}, the topological entropy $h(\G,Q)$ is independent on the
precise choices of $Q$ as monotonous maps on the subdivision arcs of $\G$. Then using Proposition \ref{Do4}, we have $h(\G,Q)=h(T,g)$.
\end{proof}

 \subsection{The entropies $h(G,L)$ and $h(\G,Q)$ are equal}\label{9}

 To complete the proof of Theorem \ref{Thurston-algorithm}, it remains to show that $h(\G,Q)=h(G,L)$.
The idea of its proof  is similar to that of Proposition \ref{equal}.
Recall that $E_G^{\text{col}}=\{e\in E_G\mid \Phi(e)\text{ is a singleton}\}$. 

\begin{proposition}\label{collapsing-edge}
For each $e\in E_G^{\rm col}$,  $L(e)$ is either one point or the union of edges in $E_G^{\rm col}$.
\end{proposition}
\begin{proof}
Let $e\in E_G^{\text{col}}$ and let $\delta\subset e$ be a subdivision arc of $G$. Then $\Phi(\delta)$ is a singleton. By Lemma \ref{commutative2}, we have that $\Phi\big(L(\delta)\big)=g\big(\Phi(\delta)\big)$ is a singleton. It follows that $L(\delta)$ is either a point or an edge of $G$ belonging to $E_G^{\text{col}}$.
\end{proof}

With this proposition, if we arrange the edges of $G$ in the order $E_G^{\text{col}}, E_G\setminus E_G^{\text{col}}$, the incidence matrix of $(G,L)$ takes the form
\[D_{(G,L)}=\left(\begin{array}{cc}
X&\star\\
\mathbf{0}&C
\end{array}\right)
\]
where $\mathbf{0}$ is a zero matrix.
Note that the matrix $C$  is exactly the incidence matrix of $(\G,Q)$, so it is enough to prove $\rho(X)=1$.

From the modification of $\Phi$ in Lemma \ref{commutative2}, we know that an edge $e(x,y)\in E_G^{\text{col}}$ if and only if $\Phi(x)=\Phi(y)$, or equivalently, $\g_g(x)=\g_g(y)$. And by the definition $\phi$, this happens if and only if $\g_f(x)$ and $\g_f(y)$ are contained in an essential $\sim$-class (see Definition \ref{essential}).

From now on and to the end, all argument take place in the $f$-plane, so we omit the subscript $f$ in all quantities for simplicity.

Let $K$ be a connected subset of $K_f$. An edge $e(x,y)$ of $G$ is called  \emph{corresponding} to $K$ if $\g(x),\g(y)\in K$. With this concept, any edge in $E_G^{\text{col}}$ corresponds to an essential $\sim$-class.

\begin{lemma}\label{correspond}
Let $e=e(x,y)\in E_G^{\text{col}}$ correspond to a connected subset $K$ of an essential $\sim$-class. Then all edges of $G$ contained in $L(e)$ correspond to $f(K)$.
\end{lemma}
\begin{proof}
Let the ordered pair $x,y$ have the separation set $(k_1,\ldots,k_p)$ with $p\geq0$, i.e., the leaf $\ov{xy}$ successively crosses $hull(\Theta_{k_1}),\ldots hull(\Theta_{k_p})$ from $x$ to $y$ with $\Theta_{k_1},\ldots,\Theta_{k_p}\in \Th$, and denote the subdivision arcs of $G$ contained in $e(x,y)$ by $\delta_i, i=0,\ldots p,$ with $x\in\delta_0$ and $y\in\delta_{p+1}$.

Since $K$ is allowable connected, then $[\g(x),\g(y)]\subset K$.
In the proof of Lemma \ref{injective-g}, we have shown that $[\g(x),\g(y)]$ intersects $\RRR(\Theta_{k_i})$ for all $i\in\{1,\ldots,p\}$. If $\Theta_{k_i}=\Theta_j(c_j)$ for some $\Theta_j(c_j)\in\Th$, then $[\g(x),\g(y)]\cap \RRR(\Theta_{k_i})=\{c_j\}$ and $\g(\theta_i)=c_j\in [\g(x),\g(y)]$ for any given $\theta_i\in \Theta_{k_i}$. If $\Theta_{k_i}=\Theta(U)$ for a critical Fatou component $U$, the intersection of $[\g(x),\g(y)]$ and $U$ consists of two internal rays of $U$, which are critical/postcritical internal rays (since $[\g(x),\g(y)]$ is in an essential $\sim$-class).
By Lemma \ref{basic-property}.(4), these two internal rays must be contained in $\RRR(\Theta_{k_i})$. We can then pick an angle $\theta_i\in\Theta_{k_i}$ satisfying $\g(\theta_i)\in[\g(x),\g(y)]$. Thus, we obtain a sequence of angles $\theta_1,\ldots,\theta_p$
 such that $\theta_i\in\Theta_{k_i}$ and $\g(\theta_i)\in [\g(x),\g(y)]\subset K$ for each $i\in\{1,\ldots,p\}$.
By setting
 $\theta_0:=x$ and $\theta_{p+1}:=y$, we have
 \begin{equation}\label{eq:2}
\g(\tau(\theta_0)),\ldots\ldots,\g(\tau(\theta_{p+1}))\in f(K).
\end{equation}
According to the definition of $L$, the edges of $G$ contained in $L(e(x,y))$ are the non-trivial arcs among $$L(\delta_0)=e(\tau(\theta_0),\tau(\theta_1)),\ldots\ldots,L(\delta_p)=e(\tau(\theta_p),\tau(\theta_{p+1})),$$
which all correspond to $f(K)$ by \eqref{eq:2}.
\end{proof}

The following lemma is quite similar to Lemma \ref{f-invariant}, both the statement and the role.

\begin{lemma}\label{eventually-period}
All edges in $E_{G}^{\rm col}$ are attracted by cycles, i.e., for each $e\in E_{G}^{\rm col}$, its iterations by $L$ eventually fall on the union of cycles in $E_{G}^{\rm col}$.
\end{lemma}
\begin{proof}
Given any edge $e\in E_G^{\text{col}}$  corresponding to an essential $\sim$-class $\x$. By repeated use of Lemma \ref{correspond},
we get that for all $n\geq1$, each edge of $G$ contained in $L^n(e)$ corresponds to $f^n(\x)$. By Lemma \ref{star-like}, there exists an $n_0>0$ such that for any  $n\geq n_0$, the set $f^n(\x)$ is either a periodic point in $J_f$  or a star-like tree consisting of periodic internal rays. Hence $f^n(\x)\cap J_f$ is a periodic point in $J_f$, denoted by $z_n$. We just need to prove that for any $n\geq n_0$, each edge of $G$ contained in $L^n(e)$ is periodic by the iterations of $L$.

 Let $n\geq n_0$. Denote by $A_n$ the set of external angles associated with $z_n$.
We claim that each pair of angles in $A_n$ are non-separated by $\Th$.
Suppose on the contrary that $x,y\in A_n$ are separated by $\Theta\in\Th$. Then, in the dynamical plane, the rays $\RRR(x),\RRR(y)$ are separated by $\RRR(\Theta)$. It follows that $z_n\in\RRR(\Theta)$. Since $z_n$ is periodic, the set $\Theta$ equals to $\Theta(U)$ for a critical Fatou component $U$, and one of the angles in $\Theta(U)$, say $\theta$, belongs to $A_n$. It contradicts that $\RRR(\theta)$ supports the Fatou component $U$ at $z_n$.

Let $e(x,y)$ be any edge of $G$ contained in $L^n(e)$.
We know from the first paragraph that $\g(x)$ and $\g(y)$ belong to $f^n(\x)\cap J_f=\{z_n\}$. By the claim above, we have $L(e(x,y))=e(\tau(x),\tau(y))$, which is also an edge of $G$ and corresponds to $f^{n+1}(\x)$. Since this argument holds for all sufficiently large $n$, and the angles in $A_n$ are periodic (because $z_n$ is periodic), the edge $e(x,y)$ must also be
 periodic by iterations of $L$.
\end{proof}

\begin{proposition}\label{equal-2}
The topological entropies verify $h(\G,Q)=h(G,L)$.
\end{proposition}
\begin{proof}
The effect of Lemma \ref{eventually-period} in this proof is the same as that of Lemma \ref{f-invariant} in the proof of Proposition \ref{equal}. Using a similar argument, just replacing $H_f,T$ and $f$ in the proof of Proposition \ref{equal} with $G,\G$ and $L$ respectively, we obtain the equation $h(G,L)=h(\G,Q)$. The details are omitted.
\end{proof}

\begin{proof}[\bf Proof of Theorem \ref{Thurston-algorithm}]
It follows directly from Propositions \ref{equal-0}, \ref{equal}, \ref{equal-1} and \ref{equal-2}.
\end{proof}

\vspace{1cm}

\noindent Yan Gao, \\
Mathemaitcal School  of Sichuan University, Chengdu 610064,
P. R. China. \\
Email: gyan@scu.edu.cn

\end{document}